\documentclass[final]{siamltex}

\usepackage{amscd}
\usepackage{amssymb}
\usepackage{amsmath}
\usepackage{latexsym}

\usepackage{enumerate}

\usepackage{tikz}

\newtheorem{prop}{Proposition}
\newtheorem{cor}{Corollary}

\newtheorem{remark}{Remark}

\newcommand{\tensor}{{\bf \Lambda}}    
\newcommand{\tensorDEn}{\Lambda_{\Du,\Eu}} 
\newcommand{\tensorDE}{\Lambda_{\Du,\Eu}} 
\newcommand{\Qt}{Q_T}    

\newcommand{\ant}{\text{int}}    
\newcommand{\ext}{\text{ext}}  

\newcommand{\diam}{\text{diam}}

\newcommand{\R}{\mathbb{R}}

\newcommand{\n}{\textbf{n}}

\newcommand{\h}{h}                 
\newcommand{\dt}{\delta t}                 

\newcommand{\W}{{\bf w}}                        
\newcommand{\cc}{\mathfrak{C}}                        
\newcommand{\cd}{\mathfrak{D}}
\newcommand{\C}{\mathcal{C}}                        
\newcommand{\cte}{\mathit{C}}

\newcommand{\K}{K}

\newcommand{\tend}{\longrightarrow}

\newcommand{\KT}{\kappa_\T}  

\newcommand{\Du}{D}                        
\newcommand{\hdt}{{\dt,\h}}                        
\newcommand{\Eu}{E}                        

\newcommand{\abs}[1]{\ensuremath{\left|#1\right|}}

\newcommand{\dm}{d}
\newcommand{\V}{\textbf{V}}           

\newcommand{\U}{\textbf{U}}              
\newcommand{\ga}{\gamma}                  
\newcommand{\pc}{\mathcal{P}_c}

\newcommand{\F}{\mathcal{F}} 

\newcommand{\X}{X} 

\newcommand{\va}{\varphi} 

\newcommand{\E}{\mathcal{E}}

\newcommand{\Q}{Q}

\newcommand{\Kr}{\delta_{\Du\Eu}}

\newcommand{\kr}{\delta}

\newcommand{\D}{\mathcal{D}}
\newcommand{\T}{\mathcal{T}}
\newcommand{\B}{\mathcal{B}}
\newcommand{\norm}[1]{\ensuremath{\left\|#1\right\|}}
\newcommand{\di}{\mathrm{div}}
\newcommand{\dd}{\mathrm{d}}
\newcommand{\gas}{g} 
\newcommand{\liq}{l}

\newcommand{\pas}{\delta t}

\newcommand{\DE}{{\Du,\Eu}}

\newcommand{\Ne}{\mathcal N}

\newcommand{\desn}{\delta_{D|E}} 
\newcommand{\de}{\delta_{D|E}^{n}}
\newcommand{\dis}{d_{\Du,\Eu}}
\newcommand{\sig}{\sigma_{\Du,\Eu}}
\newcommand{\tokl}{\tau_{\Du|\Eu}}

\newcommand{\sn}{\sum_{n=1}^{N}}
\newcommand{\sk}{\sum_{K \in \mathcal{T}_\h}}

\newcommand{\snDE}{\sn \dt \sum_{\Du \in \D_\h} \sum_{\Eu \in \D_\h}}

\newcommand{\sm}{\sn\delta t\sum_{\Du \in \D_\h} \sum_{\Eu \in \Ne(\Du)}}

\newcommand{\slmin}{s_{l,min}^{n}}

\newcommand{\sli}{(s^I_{l,\Du})^{n}}
\newcommand{\sgi}{(s^I_{g,\Du})^{n}}

\title{A combined finite volume--nonconforming finite element scheme for compressible two phase flow in porous media}

\author{Bilal Saad\thanks{King Abdullah University of Science and Technology \newline 
	     Applied Mathematics and Computational Science \newline
	     Thuwal 23955-6900, Kingdom of Saudi Arabia.({\tt bilal.saad@kaust.edu.sa}).}
\and {Mazen Saad \thanks{
         Ecole Centrale de Nantes,
         D\'epartement d' Informatique et Math\'ematiques,
         Laboratoire de Math\'ematiques Jean Leray (UMR 6629 CNRS),
         1, rue de la No\'e, BP 92101, France, ({\tt Mazen.Saad@ec-nantes.fr}).} }}

\begin{document}

\maketitle

\begin{abstract} 
We propose and analyze a combined finite volume--nonconforming 
finite element scheme on general meshes to simulate the two compressible phase 
flow in porous media. The diffusion term, which can be anisotropic and heterogeneous, 
is discretized by piecewise linear nonconforming triangular finite elements.  
The other terms are discretized by means of a cell-centered finite volume scheme on a dual mesh, where the dual volumes are constructed around the sides of the original mesh.
The relative permeability of each phase is decentred according the sign of the velocity 
at the dual interface. 
This technique also ensures the validity of the discrete maximum principle 
for the saturation under a non restrictive shape regularity of the space mesh and the positiveness of all transmissibilities.  
Next, a  priori estimates on the pressures and a function of the
saturation that denote capillary terms are established.  These stabilities results lead to some compactness arguments based on the use of the Kolmogorov compactness theorem, and allow us to derive 
the convergence of a subsequence  of the sequence of approximate solutions 
to a weak solution of the continuous equations, 
provided the mesh size tends to zero. The proof is given for the complete
system when the density of the each phase depends on the own pressure.

\end{abstract}


\pagestyle{myheadings}
\thispagestyle{plain}
\markboth{B. SAAD AND M. SAAD}{COMBINED FINITE VOLUME NONCONFORMING FINITE ELEMENT } 

\section{Introduction}\label{sec:introduction}

The simultaneous flow of immiscible fluids in porous media occurs in a 
wide variety of applications. A large variety of methods have been proposed 
for the discretization of degenerate parabolic systems modeling the 
displacement of immiscible incompressible two-phase flows in porous media. 
We refer  to \cite{aziz} and \cite{peaceman} for  the finite difference method. 
The finite volume methods have been proved to be well adapted to 
discretize conservative equations.
The cell-centered finite volume scheme has been studied e.g.
 by \cite{michel2003}, \cite{Eymard00} and \cite{brenier}. 
 Recently, the convergence analysis of a finite volume scheme
for a degenerate compressible and immiscible flow in porous media 
 has been studied by Bendahmane et al.  \cite{ZS11fv} 
 when the densities of each phase depend on the global pressure, 
 and by B. Saad and M. Saad \cite{saadsaadFV} for the complete 
 system when the density of the each phase depends on the own pressure. In these works, the medium is considered homogeneous, the permeability tensor is  proportional to the matrix identity and the mesh is supposed to be admissible in the sense that satisfying the orthogonal property as in  \cite{Eymard:book}.
 The cell-centered finite volume method with an upwind discretization 
 of the convection term ensures the stability and
 is extremely robust and have been used in industry because
 they are cheap, simple to code and robust.
 However, standard finite volume schemes do not permit to handle anisotropic diffusion 
 on general meshes see e.g.  \cite{Eymard:book}. 

Various multi-point schemes where the approximation of the flux 
through an edge involves several scalar unknowns have been proposed, 
see e.g. Coudi\`ere et al. \cite{Coudiere-al}, Eymard et al. \cite{Eymard-2004}, 
or Faille \cite{Faille}.  However, such schemes require using more points than the 
classical 4 points for triangular meshes and 5 points for quadrangular meshes in 
space dimension two, making the schemes less robust. 
%

%
On the other hand finite element method allows a very 
simple discretization of the diffusion term with a full tensor 
and does not impose any restrictions on the meshes,
they were used a lot for the discretization of  a degenerate 
parabolic problems modeling of contaminant transport in porous media.
The mixed finite element method by Dawson \cite{Dawson1}, 
the conforming piecewise linear finite element method has been studied 
e.g. by Barrett and Knabner \cite{Barrett--Knabner}, Chen and Ewing \cite{Chen-Ewing}, 
Nochetto et al. \cite{Nochetto-al}, and Rulla et al. \cite{Rulla-Walkington}. 
However, it is well-known that numerical instabilities may 
arise in the convection-dominated case. 
To avoid these instabilities, the theoretical analysis
of the combined finite volume--finite element method has been carried out for
the case of a degenerate parabolic problems with a full diffusion tensors. 
The combined finite volume--conforming finite element method
proposed and studied by Debiez et al. \cite{Debiez-al} or Feistauer et al. \cite{Feistauer-al} 
for fluid mechanics equations, are indeed quite efficient. 

This ideas is extended by  \cite{Eymard-Hilhorst-Vohralik} for the degenerate parabolic problems, 
to the combination of the mixed-hybrid finite element and finite volume methods, to
inhomogeneous and anisotropic diffusion--dispersion tensors, to space dimension
three, and finally to meshes only satisfying the shape regularity condition.
In order to solve this class of equations, Eymard et al. 
\cite{Eymard-Hilhorst-Vohralik} discretize the diffusion term by means of piecewise linear 
nonconforming (Crouzeix--Raviart) finite elements over a triangularization
of the space domain, or using the stiffness matrix of the hybridization of the lowest order
Raviart--Thomas mixed finite element method. The other terms are discretized by means
of a finite volume scheme on a dual mesh, with an upwind discretization of the convection term 
to ensures the stability, where the dual volumes are constructed around
the sides of the original triangularization. 
The intention of this paper is to extend these ideas to a fully nonlinear 
degenerate parabolic system modeling immiscible gas-water displacement in porous media 
without simplified assumptions on the state law of the density of each phase, 
to the combination of the nonconforming finite element and finite volume methods, to
inhomogeneous and anisotropic permeability tensors, to space dimension
three, and finally to meshes only satisfying the shape regularity condition.

Following \cite{Eymard-Hilhorst-Vohralik}, let us now introduce the combined scheme that we analyze in this paper. We
consider a triangulation of the space domain consisting of simplices (triangles in
space dimension two and tetrahedra in space dimension three). We next construct
a dual mesh where the dual volumes are associated with the sides (edges or faces).
To construct a dual volume, one connects the barycentres of two neighboring simplices
through the vertices of their common side.We finally place the unknowns in
the barycentres of the sides.  The diffusion term, which can be anisotropic and heterogeneous, 
is discretized by piecewise linear nonconforming triangular finite elements.  
The other terms are discretized by means of a cell-centered finite 
volume scheme on a dual mesh,
where the dual volumes are constructed around the sides 
of the original mesh, hence we obtain the combined scheme. To ensures the stability,
the relative permeability of each phase is decentred according the sign of the velocity 
at the dual interface. This technique ensures the validity of the discrete maximum principle 
for the saturation  in the case where all transmissibilities are non-negative.

This paper deals with construction and convergence analysis of a combined finite volume--nonconforming finite element 
for two compressible and immiscible flow in porous media without 
simplified assumptions on the state law of the density of each phase.
%
 The analysis of this model is based on new energy estimates on the velocities of each phase. Nevertheless, 
these estimates are degenerate in the sense that they do
not permit the control of gradients of pressure of each phase, especially when a phase is not locally present in the domain. The main idea consists to
derive from degenerate estimates on pressure of each phase, 
which not allowed straight bound on pressures, an estimate on global pressure 
and degenerate capillary term in the whole domain regardless of the presence or the disappearance of the phases.
These stabilities results with a priori estimates on the pressures and a function of the
saturation that denote capillary terms and some compactness arguments based on
the use of the Kolmogorov relative compactness theorem, allow us to derive 
the convergence of both these approximations to a weak solution of the continuous
problem in this paper provided the mesh size tends to zero. 

The organization of this paper is as follows. 
In section \ref{sec:mathematical_formulation}, we introduce the nonlinear 
parabolic system modeling the two--compressible and immiscible fluids in a 
porous media and we state the assumptions on the data and present a weak formulation of the continuous problem. In section \ref{sec:combined}, we describe 
the combined finite volume--nonconforming finite element scheme 
and we present the main theorem of convergence. 
In section \ref{sec:fundamental_lemma}, we derive three preliminary 
fundamental lemmas. In fact, we present some properties of this scheme 
and we will see that we can't control the discrete gradient 
of pressure since  the mobility of each phase vanishes
in the region where the phase is missing. 
So we are going to use  the feature of global
pressure. We show that the control of velocities ensures the 
control of the global pressure and a dissipative term on saturation in the 
whole domain regardless of the presence or the disappearance of the phases.
Section \ref{sec:basic-apriori} is devoted to a maximum principle on 
saturation, a priori estimates on the discrete velocities  and existence of discrete solutions.
In section
\ref{sec:compacity}, we derive estimates on difference of
time and space translates for the approximate solutions. In section
\ref{sec:limite} using
the Kolmogorov relative compactness theorem, we prove the convergence of
a subsequence of the sequence of approximate solutions to a weak solution of the
continuous problem.

\section{Mathematical formulation of the continuous problem}\label{sec:mathematical_formulation}

Let us state the physical model describing the immiscible displacement
of two compressible fluids in porous media. We consider the flow of two 
immiscible fluids in a porous medium. We focus on the water and gas phase, 
but the considerations below are also valid for a general wetting phase
and a non-wetting phase. 

The mathematical model is given by the
mass balance equation and Darcy's law for both phases $\alpha = \liq, \gas$. 
Let $T>0$ be the final time fixed, and let be $\Omega$ a bounded open subset of
$\R^\dm\ (\dm\geq1)$. We set $Q_T=(0,T)\times \Omega$, $\Sigma_T =
(0,T)\times \partial \Omega$. The mass conservation of each phase is given in $Q_T$ 
\begin{equation}\label{eq:principal}
  \phi(x)\partial_{t}( \rho_{\alpha}(p_\alpha)s_\alpha)(t,x) + \di (\rho_{\alpha}(p_\alpha) \V_{\alpha})(t,x)
  +\rho_\alpha(p_\alpha) s_\alpha  f_{P}^{~}(t,x) = \rho_\alpha(p_\alpha) s^I_\alpha f_{I}^{~}(t,x),
\end{equation}
where $\phi$, $\rho_\alpha$ and $s_\alpha$ are respectively the
porosity of the medium, the density of the $\alpha$ phase and the
saturation of the $\alpha$ phase. Here the functions $f_{I}^{~}$ and
$f_{P}^{~}$ are respectively the injection and production terms. Note
that in equation \eqref{eq:principal} the injection term is
multiplied by a known saturation $s^I_\alpha$ corresponding to the
known injected fluid, whereas the production term is multiplied by the
unknown saturation $s_\alpha$
corresponding to the produced fluid.\\
The velocity of each fluid $\V_\alpha$ is given by the Darcy law:
\begin{equation}
  \V_{\alpha}= - {\bf \Lambda}
  \frac{k_{r_\alpha}(s_\alpha)}{\mu_{\alpha}}\big( \nabla
  p_{\alpha}-\rho_\alpha(p_\alpha){\bf g}\big),\qquad\quad \alpha
  = \liq, \gas.
\end{equation}
where ${\bf \Lambda}$ is the permeability tensor of the porous medium, $k_{r_\alpha}$ the
relative permeability of the $\alpha$ phase, $\mu_\alpha$ the constant
$\alpha$-phase's viscosity, $p_\alpha$ the $\alpha$-phase's pressure
and ${\bf g }$ is the gravity term.
Assuming that the phases occupy the whole pore space, the phase
saturations satisfy
\begin{equation}\label{def:saturation}
  s_{\liq}+ s_{\gas} = 1.
\end{equation}
The curvature of the contact surface between the two fluids links the
jump of pressure of the two phases to the saturation by the capillary
pressure law in order to close the system
\eqref{eq:principal}-\eqref{def:saturation}
\begin{equation}\label{def:pression_capillaire.}
  p_c(s_\liq(t,x)) = p_{\gas}(t,x) - p_{\liq}(t,x).
\end{equation}
With the arbitrary choice of \eqref{def:pression_capillaire.} (the jump
of pressure is a function of $s_\liq$), the application $s_\liq\mapsto
p_c(s_\liq)$ is non-increasing, $(\frac{\dd p_c}{\dd s_\liq}(s_\liq) < 0,
\mbox{ for all } s_\liq \in [0,1])$, and usually $p_c(s_\liq=1)=0$
 when the wetting fluid is at its maximum
saturation. 

\subsection{Assumptions and main result}

The model is treated without simplified assumptions on the density of
each phase, we consider that the density of each phase depends on its
corresponding pressure. The main point is to handle a priori estimates
on the approximate solution. The studied system represents two kinds
of degeneracy: the degeneracy for evolution terms
$\partial_t(\rho_\alpha s_\alpha)$ and the degeneracy for dissipative
terms $\di(\rho_\alpha M_\alpha\nabla p_\alpha)$ when the saturation
vanishes. We will see in the section \ref{sec:basic-apriori} that we
can't control the discrete gradient of pressure since the mobility of
each phase vanishes in the region where the phase is missing. So, we
are going to use the feature of global pressure to obtain uniform
estimates on the discrete gradient of the global pressure and the
discrete gradient of the capillary term ${\mathcal B}$ (defined on
\eqref{def:beta}) to treat the degeneracy of this system.

Let us summarize some useful notations in the sequel. We recall the
conception of the global pressure as describe in \cite{chavent}
$$
M(s_\liq)\nabla p = M_\liq(s_\liq) \nabla p_\liq + M_\gas(s_\gas) \nabla p_\gas,
$$
with the $\alpha$-phase's mobility $M_\alpha$ and the total mobility are defined by 
$$
M_{\alpha}(s_{\alpha})=k_{r_\alpha}(s_{\alpha})/ \mu_{\alpha}, \quad
M(s_\liq) = M_\liq(s_\liq)+M_\gas(s_\gas).
$$
This global pressure $p$ can be written as
\begin{align}\label{def:pression_globale}
  p=p_\gas+\tilde{p}(s_\liq)=p_\liq+\bar{p}(s_\liq),
\end{align}
or the artificial pressures are denoted by $\bar{p}$ and $\tilde{p}$
defined by:
\begin{align}\label{def:terme_capillaire}
  \tilde{p}(s_\liq)=-\int_{0}^{s_\liq} \frac{M_{\liq}(z)}{M(z)}
  p_c^{'}(z)\dd z \text{ and } \overline{p}(s_\liq)=\int_{0}^{s_\liq}
  \frac{M_{\gas}(z)}{M(z)} p_c^{'}(z)\dd z.
\end{align}
We also define the capillary terms by
$$
\ga (s_\liq)=-\frac{M_{\liq}(s_\liq)M_{\gas}(s_\gas)}{M(s_\liq)}\frac{\dd p_c}{\dd
  s_\liq}(s_\liq)\geq 0,
$$
and let us finally define the function $\B$ from $[0,1]$ to $\R$
by:
\begin{align}\label{def:beta}
  \mathcal{B}(s_\liq)&=\int_{0}^{s_\liq}\ga(z) \dd z  =-
  \int_{0}^{s_\liq}\frac{M_\liq(z)M_\gas(z)}{M(z)}\frac{\dd p_c}{\dd
    s_\liq}(z) \dd z \notag\\
     & =- \int_{0}^{s_\gas} M_\liq(z)\frac{\dd
    \bar{p}}{\dd s_\liq}(z) \dd z  = \int_{0}^{s_\liq}
  M_\gas(z)\frac{\dd \tilde{p}}{\dd s_\liq}(z) \dd z.
\end{align}
Using these notations, we derive the fundamental relationship  between the velocities  and the global pressure and the capillary term:
\begin{equation}\label{plpgsbeta}
M_l(s_l)\nabla p_l=M_l(s_l)\nabla p +\nabla \mathcal{B}(s_l),\quad M_g(s_g)\nabla p_g=M_g(s_l)\nabla p -\nabla \mathcal{B}(s_l).
\end{equation} 

As mentioned above that the mobilities vanish and consequently the control of the gradient of the pressure of each phase is not possible. A main point of the paper is to give sense for the term $\nabla p_\alpha$, $\alpha=l,g$. This term is a distribution and it is not enough to give a sense of the velocity. Our approach is based on the control of the velocity of each phase. Thus the gradient of the global pressure and the function $\B$ are bounded which give a rigorous justification for the degenerate problem.

We complete the description of the model \eqref{eq:principal}
by introducing boundary conditions and initial conditions.  To the system
\eqref{eq:principal}--\eqref{def:pression_capillaire.} we add
the following mixed boundary conditions. We consider the boundary
$\partial \Omega=\Gamma_\liq\cup \Gamma_{\emph{imp}}$, where $\Gamma_\liq$
denotes the water injection boundary and $\Gamma_{\emph{imp}}$ the
impervious one.
\begin{equation}\label{cd:bord.}
  \left\{ 
    \begin{aligned}      
      p_\liq(t,x) = p_\gas(t,x)=0 & \text{ on } (0,T)\times\Gamma_\liq,  \\
      \rho_\liq \V_\liq \cdot \textbf{n} = \rho_\gas \V_\gas \cdot \textbf{n}
      = 0 & \text{ on } (0,T)\times\Gamma_{imp},
    \end{aligned} 
  \right.
\end{equation}
where $\textbf{n}$ is the outward normal to $\Gamma_{imp}$.

The initial conditions are defined on pressures
\begin{equation}\label{cd:initial.}
  p_{\alpha}(t=0) =p^{0}_{\alpha} \text{ for } \alpha=\liq,\gas  \text{ in } \Omega.
\end{equation}


Next we introduce some physically relevant assumptions on the
coefficients of the system.
\begin{enumerate}[({${A}$}1)]
\item \label{hyp:A1} There is two positive constants $\phi_{0}$ 
                     and $\phi_{1}$ such that $\phi_{0}\leq \phi(x) \leq 
                     \phi_{1}$ almost everywhere $x\in \Omega$.
\item \label{hyp:A2} $\tensor_{ij} \in L^{\infty}(\Omega)$, 
                          $\abs{\tensor_{ij}}\le \frac{C_\tensor}{d} 
                          \text{ a.e. in } \Omega$, $1 \le i,j\le d$, $C_\tensor >0$,
                          $\tensor$ is a symmetric and there exist
                          a constant $c_\tensor>0$ such that                       
                          $$
                          \left<\tensor(x) \xi,\xi \right> \geq c_\tensor
                          | \xi |^{2}, \forall \xi \in \R^d.
                          $$
\item \label{hyp:A3} The functions $M_\liq$ and $M_\gas$ belongs to 
                     ${\mathcal C}^{0}([0,1],\R^{+})$, $ M_{\alpha}(s_{\alpha}=0)=0.$ In addition,
  there is a positive constant $m_{0}>0$ such that for all $s_\liq\in
  [0,1]$,
  $$
  M_\liq(s_\liq) + M_\gas(s_\gas)\geq m_{0}.
  $$
\item \label{hyp:A4} $(f_{P}^{~},f_{I}^{~})\in (L^2(Q_T))^2$,
  $f_{P}^{~}(t,x)$, $f_{I}^{~}(t,x) \ge 0$
  almost everywhere $(t,x)\in Q_T$.
\item \label{hyp:A5} The density $\rho_{\alpha}$ is ${\mathcal
    C}^{1}(\R)$, increasing and there exist two positive constants
  $\rho_{m}>0$ and $\rho_{M}>0 $ such that $0<\rho_{m}\leq
  \rho_{\alpha}(p_{\gas})\leq \rho_{M}.$
\item \label{hyp:A6} The capillary pressure fonction $p_c\in
  \mathcal{C}^{1}([0,1];\R^{+})$, decreasing and there exists
  $\underline{p_c}>0$ such that $0<\underline{p_c}\leq |\frac{\dd
    p_c}{\dd s_\liq}|$.
\item \label{hyp:A7} The function $\ga \in C^{1}\left([0,1];\R^{+}
  \right)$ satisfies $\ga(s_\liq)>0$ for $0<s_\liq< 1$ and
  $\ga(s_\liq=1)=\ga(s_\liq=0)=0.$ We assume that $\mathcal{B}^{-1}$ (the
  inverse of $\mathcal{B}(s_\liq)=\int_{0}^{s_\liq}\ \ga(z) \dd z$) is a
  H\"{o}lder\footnote{This means that there exists a positive constant
    $c$ such that for all $a, b \in [0,\mathcal{B}(1)],$ one has
    $|\mathcal{B}^{-1}(a)-\mathcal{B}^{-1}(b)|\leq c|a -
    b|^{\theta}$.}  function of order $\theta$, with $0<\theta\leq 1,
  \text{ on } [0,\mathcal{B}(1)]$.
\end{enumerate} 
The assumptions ({A}\ref{hyp:A1})--({A}\ref{hyp:A7}) are classical for
porous media. Note that, due to the boundedness of the capillary
pressure function, the functions $\tilde{p}$ and $\bar{p}$ defined in
\eqref{def:terme_capillaire} are bounded on $[0,1]$. \\
We now give the definition of a weak solution of the problem \eqref{eq:principal}--\eqref{def:pression_capillaire.}.

\begin{definition} \label{def:weak solution} 
  $\left(\text{Weak solutions} \right).$ Under assumptions 
  ({A}\ref{hyp:A1})-({A}\ref{hyp:A7}) and suppose ($p^{0}_\liq,\ p^{0}_\gas$) belong to  $(L^2(\Omega))^2$ and  $0\le s^0_\alpha(x)\le 1$ almost everywhere in
  $\Omega$, then the pair $\left(p_\liq, p_\gas \right)$ is a weak
  solution of problem \eqref{eq:principal} satisfying :
\begin{align}
  & p_\alpha \in L^2(Q_T),~0\leq s_\alpha(t,x)\leq 1 \text{ a.e in } Q_T, \; (\alpha=\liq,\gas),\\
& p\in L^{2}(0,T;H^{1}(\Omega)),\; \mathcal{B}(s_\liq)\in L^{2}(0,T;H^{1}_{\Gamma_\liq}, (\Omega)), \\
&M_\alpha(s_\alpha)\nabla p_\alpha \in (L^2(Q_T))^\dm, 
\end{align}
such that for all $\varphi,\, \psi \in
\mathbb{C}^1([0,T];H^1_{\Gamma_\liq}(\Omega))\, \text{ with } \,
\varphi(T,\cdot)=\psi(T,\cdot)=0$,
\begin{align}
  & - \int_{Q_T} \phi \rho_\liq(p_\liq) s_\liq \partial_t\varphi\dd x \dd t
    - \int_\Omega \phi (x) \rho_\liq(p_\liq^0(x)) s_\liq^0(x)\varphi(0,x)\dd x \notag\\
  & + \int_{Q_T} \rho_\liq(p_\liq) M_\liq(s_\liq) \tensor \nabla p_\liq \cdot\nabla
  \varphi \dd x \dd t
  -\int_{Q_T}\tensor M_\liq(s_\liq)\rho_\liq^2(p_\liq) {\bf g} \cdot\nabla \varphi \dd x \dd t\label{eq:pl}\\
  &+\int_{Q_T}\rho_\liq(p_\liq)s_\liq f_{P}^{~}\varphi \dd x\dd t\notag =
    \int_{Q_T}\rho_\liq(p_\liq) s^I_\liq f_{I}^{~}\varphi \dd x \dd t,
\end{align}
\begin{align}
  &-\int_{Q_T} \phi \rho_\gas(p_\gas) s_\gas \partial_t\psi\dd x \dd t
  -\int_\Omega \phi (x) \rho_\gas(p_\gas^0(x)) s_\gas^0(x)\psi(0,x)\,dx \notag\\
  &+\int_{Q_T}M_\gas(s_\gas)\rho_\gas(p_\gas) \tensor \nabla p_\gas \cdot\nabla \psi
  \dd x \dd t
  -\int_{Q_T} \tensor M_\gas(s_\gas)\rho_\gas^2(p_\gas){\bf g} \cdot\nabla \psi \dd x \dd t\label{eq:pg}\\
  &+\int_{Q_T}\rho_\gas(p_\gas)s_\gas f_{P}^{~}\psi \dd x \dd t\notag
  =\int_{Q_T}\rho_\gas(p_\gas) s^I_\gas f_{I}^{~}\psi \dd x \dd t.
\end{align}
\end{definition}

\section{Combined finite volume--nonconforming finite element scheme} \label{sec:combined}
We will describe the space and time discretizations, 
define the approximation spaces, and introduce the combined 
finite volume--nonconforming finite element scheme in this section.
\subsection{Space and time discretizations} \label{subsec:discretizations}
In order to discretize the problem \eqref{eq:principal},
 we perform a triangulation $\T_h$ of the domain $\Omega$, consisting 
 of closed simplices such that $\overline{\Omega}=\cup_{\K \in \T_h} \K$ 
 and such that if $\K, \L \in \T_h$, $\K\ne \L$, then $\K \cap \L$ is 
 either an empty set or a common face, edge, or vertex of $\K$ and $\L$.
 We denote by $\E_h$ the set of all sides, by $\E_h^{\ant}$ the set of all
 interior sides, by $\E_h^{\ext}$ the set of all exterior sides, and by 
 $\E_\K$ the set of all the sides of an element $\K \in \T_h$. We define 
 $\h := max\{\diam(K),K\in \T_h \}$ and make the following shape regularity 
 assumption on the family of triangulations $\{ \T_h \}_h$: \\
 There exists a positive constant $\kappa_\T$ such that
 \begin{equation}\label{eq:regularity mesh}
  \min_{\K \in \T_h} \frac{\abs{K}}{\diam(\K)^\dm} \ge \kappa_T, \quad \forall \h >0. 
\end{equation}

Assumption \eqref{eq:regularity mesh} is equivalent to the 
more common requirement  of the existence of a constant $\theta_\T>0$ such that 
\begin{equation}\label{eq:regularity}
  \max_{\K \in \T_h} \frac{\diam(\K)}{\mathfrak{D}_\K} \ge \kappa_T, \quad \forall \h >0, 
\end{equation}
where $\mathfrak{D}_\K$ is the diameter of the largest ball 
inscribed in the simplex $\K$. \\ 
\bigskip
\begin{figure}[htbp]
\begin{center}
\begin{tikzpicture} [remember picture]
  \draw[thick] (-6,-0.4)--(-0.5,-0.8)--(4,0.1)--(0.5,2)--cycle;
    \filldraw [fill=gray!10,dashed] (-2,0.266)--(-0.5,-0.8)-- (0.5,2);
    \filldraw [fill=gray!10,dashed] (-0.5,-0.8)-- (0.5,2)-- (1.33,0.433);
    \filldraw [fill=gray!70,dashed] (1.33,0.433)--(-0.5,-0.8)-- (4,0.1);
  \draw[black] (-0.5,-0.8)--(0.5,2);     
   \draw [dashed] (-2,0.266)--(0.5,2);
   \node (centre) at (-2.7,0.2){$\K$};   
  \node (centre) at (1.8,0.9){$\L$};
   \draw [dashed] (1.33,0.433)--(-0.5,-0.8);
   \draw [dashed] (1.33,0.433)--(0.5,2);
\node (centre) at (1.6,0.){$\Eu$};   	
  \draw [dashed] (1.33,0.433)--(4,0.1);
   \node (centre) at (-1.,0.4){$\Du$}; 
  \node (centre) at (0,0.6){$\times$};  
  \node (centre) at (0.35,0.6){$\Q_\Du$};  
    \node (centre) at (0.48,1.25){$\sigma_\Du$};  
  \node (centre) at (1.75,-0.35){$\times$};
  \node (centre) at (1.75,-0.57){$\Q_\Eu$};  
  \node (centre) at (2.8,-0.3){$\sigma_\Eu$};  
  \node (centre) at (0.48,1.25){$\sigma_\Du$};  
  \node (centre) at (0.8,-0.3){$\sig$};  
\end{tikzpicture}
\caption{\small Triangles $\K$,$\L \in \T_\h$ and dual volumes $\Du$,$\Eu\in \D_\h$ associated with edges $\sigma_\Du,\sigma_\Eu \in \E_\h$}
\label{domaine}
\end{center}
\end{figure}
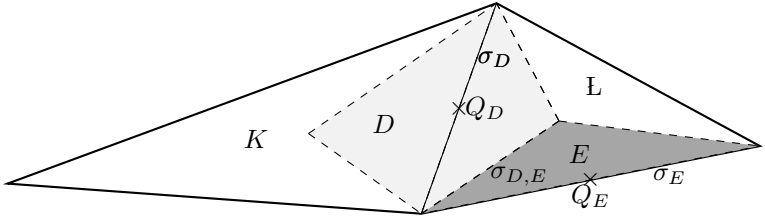

 We also use a dual partition $\D_\h$ of $\Omega$ such that 
  $\overline{\Omega} = \cup_{\Du \in \D_h} \Du$. There is one dual element 
  $\Du$ associated with each side $\sigma_\Du \in \E_\h$. We construct it by
  connecting the barycentres of every $\K \in \T_\h$ that contains $\sigma_\Du$  
  through the vertices of $\sigma_\Du$. For $\sigma_\Du \in \E^\ext_\h$, the contour 
  of $\Du$ is completed by the side $\sigma_\Du$ itself. We refer to Fig. \ref{domaine} 
  for the 
  two-dimensional case. We denote by $\Q_\Du$ the barycentre of the side $\sigma_\Du$. 
  As for the primal mesh, we set $\F_\h$, $\F_\h^\ant$, $\F_\h^\ext$ and $\F_\Du$ for 
  the dual mesh sides. We denote by $\D_\h^\ant$ the set of all interior and by $\D_\h^\ext$
  the set of all boundary dual volumes. We finally denote by $\Ne(\Du)$ the set of all adjacent volumes
  to the volume $\Du$, 
  $$
  \Ne(\Du) := \{  \Eu \in \D_\h; \exists \sigma \in \F_\h^\ant \text{ such that } \sigma = \partial \Du \cap \partial \Eu \}
  $$
  and remark that 
  \begin{align}\label{KcapD}
  	\abs{\K \cap \Du} = \frac{\abs{K}}{ \dm + 1},
  \end{align}
  for each $\K \in \T_\h$ and $\Du \in \D_\h$ such that $\sigma_\Du \in \E_\K$. 
  For $\Eu \in \Ne(\Du)$, we also set $\dis := \abs{\Q_\Eu - \Q_\Du}$,  $\sig := \partial \Du \cap \partial \Eu$
  and $\K_\DE$ the element of $\T_\h$ such that $\sig \subset \K_\DE$. \\
  
  The problem under consideration is time-dependent, hence we also need
   to discretize the time interval $(0,T)$. The time discretization of
  $(0,T)$ is given by an integer value $N$ and by a strictly
  increasing sequence of real values $(t^n)_{n\in [0,N]}$ with
  $t^0=0$ and $t^{N} = T$. Without restriction, we consider a uniform step time  $\dt = t^{n} - t^{n-1}$, for $n\in [1,N]$.\\
  
  We define the following finite-dimensional spaces:
\begin{align*}
  \X_\h & :=   \{     \va_\h \in L^2(\Omega); \va_\h |_\K \text{ is linear } \forall \K \in \T_\h,  \\ &
  	     \qquad  \va_\h \text{ is continuous at the points } \Q_\Du, \Du \in \D_\h^\ant\},  \\ 
\X_\h^0 & :=  \{ \va_\h \in \X_\h; \va_\h(\Q_\Du) = 0 \quad \forall \Du \in \D_\h^\ext  \}.	     
  \end{align*}
  The basis of  $\X_\h$ is spanned by the shape functions $\va_\Du$, $\Du \in \D_\h$, such that $\va_\D(\Q_\Eu) = \Kr$,
  $\Eu \in \D_\h$, $\kr$ being the Kronecker delta. We recall that the approximations in these spaces are nonconforming since
  $\X_\h \not\subset H^1(\Omega)$. We equip $\X_\h$ with the seminorm 
  $$
  \norm{u_\h}^2_{\X_\h} :=  \sum_{\K \in \T_\h} \int_{\K} \abs{ \nabla u_\h }^2 \dd x, 
  $$
  which becomes a norm on $\X_\h^0$. 

%
%
For a given  value $u_\Du, \Du\in \D_\h$  (resp. $u_\Du^n, \Du\in \D_\h, n\in[0,N]$),  we define a  constant piecewise function  as : $u(x)=u_\Du$ for $x\in \Du$ (resp. $u(t,x) = u_\Du^n$ for $x\in \Du$, $t\in ]t^{n-1}, t^{n}]$. Next, we define the discret differential operateur : 
$$
\desn(u) = u_E-u_D, \text{ and }   \de(u) = u_E^n-u_D^n.
$$
\begin{definition} \label{approximationsolution}
	Let the values $u_\Du^n$, $\Du \in \D_\h$, $n \in \{0,1, \cdots, N\}$. 
	As the approximate solutions of the problem 
	by means of the combined finite volume--nonconforming finite element scheme, we understand: 
	\begin{enumerate}
		\item A function $u_{\dt,\h}$ such that  
		\begin{align}\label{sol:FE}
			& u_{\dt,\h} (x,0) = u_\h^0(x) \text{ for } x \in \Omega, \notag \\ 
			& u_{\dt,\h} (x,t) = u_\h^n(x) \text{ for } x \in \Omega, t \in (t_{n-1},t_n] \qquad n \in \{1, \cdots, N\} ,
		\end{align}
		where $u_\h^n = \sum_{\Du \in \D_\h} u_\Du^n \va_\Du$;
		\item A function $\tilde{u}_{\dt,\h} $ such that 
		\begin{align}\label{sol:FV}
			& \tilde{u}_{\dt,\h}(x,0) = u_\Du^0 \text{ for }  x \in  \mathring{\Du}, \Du \in \D_\h, \notag \\ 
		  	& \tilde{u}_{\dt,\h}(x,t) = u_\Du^n \text{ for }  x \in  \mathring{\Du}, \Du \in \D_\h, t \in (t_{n-1},t_n] \qquad n \in \{1, \cdots, N\}.
		\end{align}		
	\end{enumerate}		
\end{definition}

The function $u_{\dt,\h}$ is piecewise linear and continuous in the barycentres of the interior sides
in space and piecewise constant in time; we will call it a \emph{nonconforming finite element solution}. The 
function $\tilde{u}_{\dt,\h}$ is given by the values of $u_\Du^n$ in side barycentres and is piecewise constant
on the dual volumes in space and piecewise  constant  in time; we will call it a \emph{finite volume solution.}

 \subsection{The combined scheme} \label{subsec:combined}
 For more clarity and for presentation simplicity, we present the combined scheme for a horizontal field and then we neglect the gravity effect. In remark \ref{rem:grav}, we indicate how to modify the scheme  to include the gravity terms.   

\begin{definition}{\bf (Combined scheme)} The fully implicit combined 
  finite volume-nonconforming finite element scheme for the problem 
  \eqref{eq:principal} reads: find the values $p_{\alpha,\Du}^n$, $\Du \in \D_\h$,
  $n \in \{1,\dotsb,N\}$, such that  
\begin{equation}\label{cd:initial discret} 
  p_{\alpha,\Du}^0=\frac{1}{\abs{\Du}} \int_{\Du} p^0_\alpha(x) \dd x,\;  
  s_{\alpha,\Du}^0=\frac{1}{\abs{\Du}} \int_{\Du} s^0_\alpha(x) \dd x, \text{ for all } 
  \Du \in \D_\h^{\ant},
\end{equation} 
\begin{multline}\label{eq:pl disc}
  \abs{\Du} \phi_\Du \frac{\rho_\liq(p^{n}_{\liq,\Du}) s^{n}_{\liq,\Du} - 
                           \rho_\liq(p^{n-1}_{\liq,\Du})   s^{n-1}_{\liq,\Du}}{\dt}      -
  \sum_{\Eu \in \Ne( \Du)} \rho^{n}_{\liq,\DE} \; M_{\liq}(s_{\liq,\DE}^{n}) \; 
                              \tensorDE \;  \de(p_\liq)   \\  + 
    \abs{\Du} \rho_\liq(p_{\liq,\Du}^{n}) s_{\liq,\Du}^{n}
   f_{P,\Du}^{n} = \abs{\Du} \rho_\liq(p_{\liq,\Du}^{n}) \sli f_{I,\Du}^{n},
\end{multline}
\begin{multline}\label{eq:pg disc}
 \abs{\Du} \phi_\Du \frac{\rho_\gas(p^{n}_{\gas,\Du}) s^{n}_{\gas,\Du} - 
                           \rho_\gas(p^{n-1}_{\gas,\Du})   s^{n-1}_{\gas,\Du}}{\dt}      -
  \sum_{\Eu \in \Ne( \Du)} \rho^{n}_{\gas,\DE} \; M_{\gas}(s_{\gas,\DE}^{n}) \; 
                              \tensorDE \; \de(p_\gas)    \\  + 
   \abs{\Du} \rho_\gas(p_{\gas,\Du}^{n}) s_{\gas,\Du}^{n}
   f_{P,\Du}^{n} = \abs{\Du} \rho_\gas(p_{\gas,\Du}^{n}) \sgi f_{I,\Du}^{n},
\end{multline}
\begin{equation}\label{eq:pc disc}
  p_c(s^{n}_{\liq,\Du}) = p_{\gas,\Du}^{n}-p_{\liq,\Du}^{n}.
\end{equation}
\end{definition}
We refer to the matrix $\tensor$ of the elements $\tensor_{\Du,\Eu}$, 
$\Du, \Eu \in \D_\h^{\ant}$, 
as to the diffusion matrix. This matrix, the stiffness matrix of the nonconforming finite
element method writes in the form
\begin{equation}\label{tensor}
\tensorDE := - \sum_{\K \in \T_\h} 
\left(\tensor(x) \nabla \varphi_\Eu, \nabla \varphi_\Du \right)_{0,\K} \quad \Du, \Eu \in \D_\h.
\end{equation}
%
%
Notice that  the source terms are, for  $n\in \{1,\ldots,N\}$
$$
f^{n}_{P,\Du}:=
          \frac{1}{\dt\abs{\Du}}
          \int_{t^{n-1}}^{t^{n}}
          \int_\Du f_P(t,x)\,dx dt,\quad f^{n}_{I,\Du}:=
          \frac{1}{\dt\abs{\Du}}
          \int_{t^{n-1}}^{t^{n}}
          \int_\Du f_I(t,x)\,dx dt. 
         $$
The mean value of the density of each phase on interfaces is not
classical since it is given as
\begin{equation}\label{meanrho}
\begin{aligned}
  \frac{1}{\rho^{n}_{\alpha,\DE}}=\begin{cases}
    \frac{1}{p_{\alpha,\Eu}^{n}-p_{\alpha,\Du}^{n}}
    \int_{p_{\alpha,\Du}^{n}}^{p_{\alpha,\Eu}^{n}}
    \frac{1}{\rho_\alpha(\zeta)}\,d\zeta
    & \text{ if } p_{\alpha,\Du}^{n} \ne p_{\alpha,\Eu}^{n}, \\
    \frac{1}{\rho^{n}_{\alpha,\Du}} & \text{ otherwise}.
          \end{cases}
\end{aligned}
 \end{equation}
 This choice is crucial to obtain estimates on discrete pressures. 
  
We denoted
  \begin{equation}
  \label{Gupwind}
   G_\alpha(s^{n}_{\alpha,\Du},s^{n}_{\alpha,\Eu}; \de(p_{\alpha})) = -M_{\alpha}(s_{\alpha,\DE}^{n}) \de(p_{\alpha}),
\end{equation}
the numerical fluxes, where $M_{\alpha}(s_{\alpha,\DE}^{n})$ denote the upwind
discretization of $M_\alpha(s_\alpha)$ on the interface $\sigma_{\DE}$
and
\begin{align}\label{notation:saturation_interface}
s_{\alpha,\DE}^{n}=\begin{cases}
  & s_{\alpha,\Du}^{n} \text{ if } (D,E)\in \E_\alpha^{n},\\
  & s_{\alpha,\Eu}^{n} \text{ otherwise, }
\end{cases}
\end{align}
with the set $\E_\alpha^{n}$ is subset of $\E_\h$
such that
\begin{align}\label{notation:set}
  \E_\alpha^{n}=\{ (\Du,\Eu)\in \E_\h,
  \tensorDE\; \de(p_\alpha) = \tensorDE
  \left( p_{\alpha,\Eu}^{n}-p_{\alpha,\Du}^{n}\right)\le 0 \}.
\end{align}

\begin{remark}\label{rem:grav}
 To take into account the gravity  term, it is enough to modify, for example,  the numerical fluxes to be 
 $$
  G_\alpha(s^{n}_{\alpha,\Du},s^{n}_{\alpha,\Eu}; \de(p_{\alpha})) =  -M_{\alpha}(s_{\alpha,\DE}^{n}) \de(p_{\alpha})+ \rho_\alpha(p^{n}_{\alpha,\DE})\Big( M_{\alpha}(s_{\alpha,\D}^{n}) g_{D|E}^+ - M_{\alpha}(s_{\alpha,E}^{n}) g_{D|E}^-\Big),
 $$
 with $g_{D|E}^\pm=(g\cdot \eta_{D|E})^\pm$. This numerical flux satisfy is consistent, conservative and monotone, thus the convergence result remains valid.
 \end{remark}

In the sequel we shell consider apart the following special case: \\
All transmissibility's are non-negative, i.e.
\begin{align}\label{matricepositive}
	\tensorDE \ge 0 \quad \forall \Du \in \D_\h^\ant, \Eu \in \Ne(\Du).
\end{align}
Since
$$
\nabla \varphi_{\Du|\K} = \frac{\abs{\sigma_\D}}{\abs{\K}}\n_{\sigma_\D},  \quad \K \in \T_\h, \sigma_\D \in \E_\K
$$
with $\n_{\sigma_\D}$ the unit normal vector of the side $\sigma_\D$ outward to $\K$, one can immediately 
see that Assumption \eqref{matricepositive} is satisfied when the diffusion tensor reduces to a scalar function and when the
magnitude of the angles between $\n_{\sigma_\D}$, $\sigma_\D \in \E_\K$, for all $\K \in \T_\h$ is greater or equal to 
$\pi/2$.

\bigskip

The main result of this paper is the following theorem.
\begin{theorem} \label{theo:principal} 
   There exists an approximate solutions
  $(p_{\alpha,D}^{n})_{n,D}$ corresponding to the system
  \eqref{eq:pl disc}-\eqref{eq:pg disc}, which converges (up to a
  subsequence) to a weak solution $p_\alpha$ of
  \eqref{eq:principal} in the sense of the Definition
  \ref{def:weak solution}.
\end{theorem}

\section{Preliminary fundamental lemmas}\label{sec:fundamental_lemma}

In this section, we will first present several technical lemmas that will be used in 
our latter analysis to obtain a priori estimate of the solution of the discrete problem.

\bigskip

\begin{lemma} \label{coercive}
	For all $u_\h = \sum_{\Du \in \D_\h} u_\Du \va_\Du \in \X_\h$ ,
	$$
	\sum_{\Du \in \D_\h} \sum_{\Eu \in \Ne(\Du)} \tensorDE (\desn{u})^2 \ge c_\tensor  \norm{u_\h}_{\X_\h}^2.
	$$
\end{lemma}
\begin{proof}
	We have  
		\begin{align*}
		\sum_{\Du \in \D_\h} \sum_{\Eu \in \Ne(\Du)} \tensorDE (\desn{u})^2 & = 
		\sum_{\Du \in \D_\h} \sum_{\Eu \in \Ne(\Du)} \tensorDE (u_\Eu - u _\Du)^2 \\ & = - 2
		\sum_{\Du \in \D_\h} u_\Du \sum_{\Eu \in \Ne(\Du)} \tensorDE (u_\Eu - u _\Du) \\ & = -2
		\sum_{\Du \in \D_\h} u_\Du \sum_{\Eu \in \D_\h} \tensorDE u_\Eu  \\ & = 2
		\sum_{\K  \in \T_\h} ( \tensor \nabla u_\h, \nabla u_\h)_{0,\K} \ge c_\tensor \norm{u_\h}_{\X_\h}^2,  
		\end{align*}
	using  \eqref{tensor} and assumption (A\ref{hyp:A2}). 
\end{proof}

\bigskip

We state this lemma without a proof as well (cf. \cite[Lemma 3.1]{Eymard-Hilhorst-Vohralik}):
 \begin{lemma}
  	For all $u_\h = \sum_{\Du \in \D_\h} u_\Du \va_\Du \in \X_\h$, one has 
  \begin{align}
      & \sum_{\sig \in \F_\h^\ant} \frac{\abs{\sig}}{\dis} \left( u_\Eu - u_\Du \right)^2 \le         \frac{\dm + 1}{2(\dm-1)\KT } \norm{u_\h}_{\X_\h}^2. \label{ineq2}
  \end{align}
  \end{lemma}

\bigskip

In the continuous
case, we have the following relationship between the global pressure,
capillary pressure and the pressure of each phase
\begin{equation}\label{eq:relation}
  \begin{aligned}
     M_{l}
    |\nabla p_{l}|^2 + M_{g} |\nabla p_{g}|^2=M |\nabla p |^{2} + \frac{M_{l}M_{g}}{M}|\nabla p_c |^{2} .
  \end{aligned}
\end{equation}
This relationship, means that, the control of the velocities ensures
the control of the global pressure and the capillary terms $\B$ in the
whole domain regardless of the presence or the disappearance of the
phases.

In the discrete case, these relationship, are not obtained in a
straightforward way. This equality is replaced by three discrete
inequalities which we state in the following lemma.

\begin{lemma}\label{lemma:preliminary_step}
$\left(\text{Total mobility, global pressure, Capillary term } \B \text{ and Dissipative terms } \right)$
  Under the assumptions $({A}\ref{hyp:A1})-({A}\ref{hyp:A7})$ and the
  notations \eqref{def:pression_globale}.  Then for all $(\Du,\Eu)\in \D$ and for all $n\in [0,N]$ the
  following inequalities hold:
  \begin{equation}\label{est:mobilite}
    M_{\liq,\Du|\Eu}^{n} + M_{\gas,\Du|\Eu}^{n} \ge m_0,
  \end{equation}

 \begin{equation}\label{ineq:pression_globale}
   m_0\Big(\de (p)\Big)^2
   \le M_{\liq,\Du|\Eu}^{n}\Big(\de
   (p_\liq)\Big)^2 + M_{\gas,\Du|\Eu}^{n}\Big(\de (p_\gas)\Big)^2.
 \end{equation}
 
 \begin{equation}\label{ineq:discrete_beta}
    (\de(\B(s_\liq)))^2  \le M_{\liq,\Du|\Eu}^{n}\Big(\de (p_\liq)\Big)^2 +
    M_{\gas,\Du|\Eu}^{n}\Big(\de (p_\gas)\Big)^2, 
  \end{equation}
  \begin{equation}\label{ineq:discrete_pbar}
    M_{\liq,\Du|\Eu}^{n}(\de(\bar{p}(s_\liq)))^2  \le
    M_{\liq,\Du|\Eu}^{n}\Big(\de (p_\liq)\Big)^2 + M_{\gas,\Du|\Eu}^{n}\Big(\de
    (p_\gas)\Big)^2,
      \end{equation}
and 
   \begin{equation}\label{ineq:discrete_ptilde}
     M_{\gas,\Du|\Eu}^{n}(\de(\tilde{p}(s_\liq)))^2  \le
     M_{\liq,\Du|\Eu}^{n}\Big(\de (p_\liq)\Big)^2 + M_{\gas,\Du|\Eu}^{n}\Big(\de
     (p_\gas)\Big)^2.
   \end{equation}

\end{lemma}

In  \cite{saadsaadFV}, the authors prove this lemma on primal mesh satisfying the orthogonal condition. This proof use only two neighbors elements and it is based only on the definition of the global pressure. Thus, we state the above Lemma without proof since the proof made in \cite{saadsaadFV} remains valid on the dual mesh.

\section{A priori estimates and existence of the approximate
  solution}\label{sec:basic-apriori}
  
We derive new energy estimates on the discrete velocities
$M_\alpha(s_{\alpha,\Du|\Eu}^{n}) \de(p_\alpha)$. Nevertheless, these
estimates are degenerate in the sense that they do not permit the
control of $\de(p_\alpha)$, especially when a phase is missing. So,
the global pressure has a major role in the analysis, we will show
that the control of the discrete velocities
$M_\alpha(s_{\alpha,\Du|\Eu}^{n}) \de(p_\alpha)$ ensures the control of
the discrete gradient of the global pressure and the discrete gradient
of the capillary term ${\mathcal B}$ in the whole domain regardless of the
presence or the disappearance of the phases.

The following section gives us some necessary energy estimates to
prove the theorem \ref{theo:principal}.
\subsection{The maximum principle}
Let us show in the following Lemma that the phase by phase upstream
choice yields the $L^\infty$ stability of the scheme which is a basis
to the analysis that we are going to perform.
\begin{lemma}\label{lem:principe_maximum}$\left(\text{Maximum principe} \right)$.
  Under assumptions ({A}\ref{hyp:A1})-({A}\ref{hyp:A7}).  Let
  $(s_{\alpha,\Du}^{0})_{\Du  \in \D_\h}\in [0,1]$ and assume that $(p_{\alpha,\Du}^{n})_{\Du  \in \D_\h}$ is a solution of
  the finite volume \eqref{cd:initial discret}-\eqref{eq:pc disc}. Then, the
  saturation $(s_{\alpha,\Du}^{n})_{\Du \in \D_\h,}$ remains in $[0,1]$ for all $\Du \in \D_\h, \; n \in
    \{1,\ldots,N\}$.
\end{lemma}
\begin{proof}
  Let us show by induction in $n$ that for all $\Du \in \D_\h, ~
  s^n_{\alpha,\Du}\geq 0$ where $\alpha=\liq,\gas$. For $\alpha=\liq$, the
  claim is true for $n=0$ and for all $\Du \in \D_\h$. We argue by
  induction that for all $\Du \in \D_\h$, the claim is true up to
  order $n$. We consider the control volume $\Du$ such that
  $s^{n}_{\liq,\Du}=\min{\{s^{n}_{\liq,\Eu}\}}_{\Eu\in
    \D_\h}$ and we seek that $s^{n}_{\liq,\Du}\geq 0$.\\
  For the above mentioned purpose, multiply the equation in
  \eqref{eq:pl disc} by $-(s_{\liq,\Du}^{n})^-$, we obtain
  \begin{multline}\label{non-negative}
    - \abs{\Du}\phi_\Du\frac{\rho_\liq(p^{n}_{\liq,\Du})
      s^{n}_{\liq,\Du}-\rho_\liq(p^{n-1}_{\liq,\Du}) s^{n-1}_{\liq,\Du}}{\dt}
    (s_{\liq,\Du}^{n})^-  \\ - \sum_{\Eu \in \Ne(\Du) } \tokl 
    \rho^{n}_{\liq,\DE} \; \tensorDE \;
    G_\liq(s^{n}_{\liq,\Du},s^{n}_{\liq,\Eu};\de(p_{\liq})) (s_{\liq,\Du}^{n})^-
    \\ - \abs{\Du}
    \rho_\liq(p_{\liq,\Du}^{n}) s_{\liq,\Du}^{n}
    f_{P,\Du}^{n}(s_{\liq,\Du}^{n})^- = -\abs{\Du} \rho_\liq(p_{\liq,\Du}^{n})
    \sgi f_{I,\Du}^{n}(s_{\liq,\Du}^{n})^-\le 0.
\end{multline}
The numerical flux $G_\liq$ is nonincreasing with respect to
$s_{\liq,\Eu}^{n}$, and consistence, we get

\begin{align}\label{nonnegat:I:2}
  G_\liq(s^{n}_{\liq,\Du},s^{n}_{\liq,\Eu};\de(p_{\liq}))\,(s_{\liq,\Du}^{n})^-
  &\le G_\liq(s^{n}_{\liq,\Du},s^{n}_{\liq,\Du}
  ;\de(p_{\liq}))\,(s_{\liq,\Du}^{n})^-\notag\\
  &=-\de(p_{\liq})\,M_\liq(s^{n}_{\liq,\Du}) \,(s_{\liq,\Du}^{n})^- = 0.
\end{align}
Using the identity
$s_{\liq,\Du}^{n}=({s_{\liq,\Du}^{n}})^+-(s_{\liq,\Du}^{n})^-$,
and the mobility $M_\liq$ extended by zero on $]-\infty, 0]$, then
$M_\liq(s^{n}_{\liq,\Du}) (s_{\liq,\Du}^{n})^- = 0$ and
\begin{multline}\label{nonnegat:I:3}
   -\abs{\Du} \rho_\liq(p_{\liq,\Du}^{n})
  s_{\liq,\Du}^{n} f_{P,K}^{n}(s_{\liq,\Du}^{n})^-  = \abs{\Du} \rho_\liq(p_{\liq,\Du}^{n})
  f_{P,K}^{n}((s_{\liq,\Du}^{n})^-)^2 \geq 0.
\end{multline}
Then, we deduce from \eqref{non-negative} that
$$  
\rho_\liq(p^{n}_{\liq,\Du})|(s^{n}_{\liq,\Du})^{-}|^2
  +\rho_\liq(p^{n-1}_{\liq,\Du})s^{n-1}_{\liq,\Du}(s_{\liq,\Du}^{n})^-\leq
0,
$$
and from the nonnegativity of $s^{n-1}_{\liq,\Du}$, we obtain
$(s_{\liq,\Du}^{n})^-=0$.  This implies that $s_{\liq,\Du}^{n}\geq
0$ and
$$
0\le s^{n}_{\liq,\Du} \le s^{n}_{\liq,\Eu} \text{ for all } n \in [0,N-1]
\text{ and } \Eu \in \D_\h.
$$ 
In the same way, we prove $s_{\gas,\Du}^{n}\ge 0$.
\end{proof}
\subsection{Estimations on the pressures}
We now give a priori estimates satisfied by the solution values 
$p_\Du^n$, $\Du \in \D_\h$, $\{1,\ldots,N\}$.
\begin{prop}\label{prop:estimation_pression}
  Let $(p_{\liq,\Du}^n,p_{\gas,\Du}^n)$ be
  a solution of \eqref{cd:initial discret}-\eqref{eq:pc disc}. Then, there exists
  a constant $C>0$, which only depends on $M_\alpha$, $\Omega$, $T$,
  $p^{0}_\alpha$, $s^0_\alpha$, $s^I_\alpha$, $f_P$, $f_I$ and not on
  $\D_\h$, such that the solution of the combined scheme satisfies 
\begin{align}\label{est:p_alpha}
  \sm \tensorDE M_\alpha(s_{\alpha,\Du|\Eu}^{n})
  \abs{p_{\alpha,\Eu}^{n}-p_{\alpha,\Du}^{n}}^2\le C,
\end{align}
and
\begin{align}\label{est:p_lobale}
  \sn \dt  \norm{p_\h}_{\X_\h}^2\le C.
\end{align}
\end{prop}
\begin{proof} We define the function $
  \mathcal{A}_{\alpha}(p_{\alpha}) :=
  \rho_{\alpha}(p_{\alpha})g_{\alpha}(p_{\alpha}) - p_{\alpha},$ $
  \pc(s_\liq) := \int_{0}^{s_\liq}p_c(z)\dd z $ and $
  g_{\alpha}(p_{\alpha}) =
  \int_{0}^{p_{\alpha}}\frac{1}{\rho_{\alpha}(z)}\dd z$. In the
  following proof, we denote by $C_i$ various real values which 
  independent on $\D$ and $n$.
  To prove the estimate
  \eqref{est:p_alpha}, we multiply\eqref{eq:pl disc} and
  \eqref{eq:pg disc} respectively by $g_\liq(p_{\liq,\Du}^{n})$, 
  $g_\gas(p_{\gas,\Du}^{n})$ and adding them, then summing the resulting  
  equation over $\Du \in \D_\h$ and $n \in \{1,\cdots, N\}$.  
  We thus get:
  \begin{equation}\label{disc_estimation}
    E_{1}+E_{2}+E_{3}= 0,
  \end{equation}
  where
  \begin{align*}
    E_{1} = \sn \sum_{\Du \in \D_\h} \abs{\Du}\phi_\Du \Big(
    (\rho_\liq(p^{n}_{\liq,\Du})s^{n}_{\liq,\Du}
    -\rho_\liq(p^{n-1}_{\liq,\Du})s^{n-1}_{\liq,\Du})\; g_\liq(p^{n}_{\liq,\Du}) \\
    + (\rho_\gas(p^{n}_{\gas,\Du})s^{n}_{\gas,\Du}
    -\rho_\gas(p^{n-1}_{\gas,\Du})s^{n-1}_{\gas,\Du})\; g_\gas(p^{n}_{\gas,\Du})
    \Big),
  \end{align*}
  \begin{align*}
    E_{2} = \sn\pas \sum_{\Du \in \D_\h} \sum_{\Eu \in \Ne(\D)} \tensorDEn
    \Big(\rho^{n}_{\liq,\DE}
    G_\liq(s^{n}_{\liq,\Du},s^{n}_{\liq,\Eu};\de(p_{\liq}))\;
    g_\liq(p_{\liq,\Du}^{n})\\
    + G_\gas(s^{n}_{\gas,\Du},s^{n}_{\gas,\Eu};\de(p_{\gas}))\;
    g_\gas(p_{\gas,\Du}^{n})
    \Big),
  \end{align*}
  \begin{align*}
    E_3 = \sn\pas \sum_{\Du \in \D_\h}\abs{\Du}
    \Big(\rho_\liq(p_{\liq,\Du}^{n}) s_{\liq,\Du}^{n} f_{P,\Du}^{n}
    g_\liq(p_{\liq,\Du}^{n}) - \rho_\liq(p_{\liq,\Du}^{n}) \sgi f_{I,\Du}^{n}
    g_\liq(p_{\liq,\Du}^{n}) \\ 
    + \rho_\gas(p_{\gas,\Du}^{n})
    s_{\gas,\Du}^{n} f_{P,\Du}^{n} g_\gas(p_{\gas,\Du}^{n}) -
    \rho_\gas(p_{\gas,\Du}^{n}) \sgi f_{I,\Du}^{n} g_\gas(p_{\gas,\Du}^{n})
    \Big).
  \end{align*}
To handle the first term of the equality \eqref{disc_estimation}, let us recall the following inequality : 
\begin{multline}
  \label{bibi12}
  \bigl(\rho_\liq(p_\liq^n)s_\liq^n-\rho_\liq(p_{\liq}^{n-1})
  s^{n-1}_{\liq}\bigr)g_\liq(p_\liq^n)+
  \bigl(\rho_\gas(p_\gas^n)s_\gas^n-\rho_\gas(p_\gas^{n-1})
  s^{n-1}_\gas\bigr)g_\gas(p_\gas^n) \\
  \ge\mathcal{A}_\liq(p_\liq^{n})s_\liq^{n}-\mathcal{A}_\liq(p_\liq^{n-1})s^{n-1}_\liq +
  \mathcal{A}_\gas(p_\gas^{n})s_\gas^{n}-\mathcal{A}_\gas(p_\gas^{n-1})s^{n-1}_\gas -
  \pc(s_\liq^{n})+\pc(s_\liq^{n-1}),
\end{multline}
using  the concavity property
of $g_\alpha$ and    $\pc$, in \cite{ZS10} the authors prove the above inequality.\\
So, this yields to 
\begin{multline}\label{E1}
  E_1 \ge \sum_{\Du \in \D_\h} 
      \phi_\Du \abs{\Du} \Big(s_{\liq,\Du}^N \mathcal{A}(p_{\liq,\Du}^N) 
                        -  s_{\liq,\Du}^0 \mathcal{A}(p_{\liq,\Du}^0)
                        +  s_{\gas,\Du}^N \mathcal{A}(p_{\gas,\Du}^N) 
                        -  s_{\gas,\Du}^0 \mathcal{A}(p_{\gas,\Du}^0)
                    \Big) \\  
 - \sum_{\Du\in \D_\h}\phi_\Du \abs{\Du} \pc(s_{\liq,\Du}^N) 
 + \sum_{\Du \in \D_\h} \phi_\Du \abs{\Du}\pc(s_{\liq,\Du}^0).
\end{multline}
Using the fact that the numerical fluxes $G_\liq$ and $G_\gas$ are
conservative, we obtain by
discrete integration by parts 
\begin{align*}
  E_2 = \frac{1}{2}\sn\pas \sum_{\Du\in \D_\h} \sum_{\Eu \in \Ne(\Du)} \tensorDEn  
        \Big(& 
              \rho_{\liq,\DE}^{n} G_\liq(s^{n}_{\liq,\Du},s^{n}_{\liq,\Eu};\de(p_{\liq})) 
                                   (g_\liq(p_{\liq,\Du}^{n})-g_\liq(p_{\liq,\Eu}^{n})) \\ &
           +  \rho_{\gas,\DE}^{n} G_\gas(s^{n}_{\gas,\Du},s^{n}_{\gas,\Eu};\de(p_{\gas}))            
                                   (g_\gas(p_{\gas,\Du}^{n})-g_\gas(p_{\gas,\Eu}^{n}))
        \Big),
\end{align*}
and due to the correct choice of the density of the phase 
$\alpha$ on each interface,
\begin{align}\label{choice_of_density}
  \rho_{\alpha,\DE}^{n} (g_{\alpha}(p_{\alpha,\Du}^{n})-g_{\alpha}(p_{\alpha,\Eu}^{n}))=
  p_{\alpha,\Du}^{n}-p_{\alpha,\Eu}^{n},
\end{align}
we obtain 
\begin{align*}
  E_2 = \frac{1}{2}\sn\pas \sum_{\Du \in \D_\h} \sum_{\Eu \in \Ne(\Du)} &\tensorDEn 
        \Big( 
             G_\liq(s^{n}_{\liq,\Du},s^{n}_{\liq,\Eu};\de(p_{\liq}))
               (p_{\liq,\Du}^{n}-p_{\liq,\Eu}^{n})\\ & +
             G_\gas(s^{n}_{\gas,\Du},s^{n}_{\gas,\Eu};\de(p_{\gas}))
               (p_{\gas,\Du}^{n}-p_{\gas,\Eu}^{n}) 
        \Big).
\end{align*}
The definition of the upwind fluxes in \eqref{Gupwind} implies 
\begin{multline*}  
  G_\liq(s^{n}_{\liq,\Du},s^{n}_{\liq,\Eu};\de(p_{\liq}))
        (p_{\liq,\Du}^{n}-p_{\liq,\Eu}^{n})
     +
  G_\gas(s^{n}_{\gas,\Du},s^{n}_{\gas,\Eu};\de(p_{\gas}))
        (p_{\gas,\Du}^{n}-p_{\gas,\Eu}^{n})
\\ =   M_{\liq}(s_{\liq,\DE}^{n})(\de(p_{\liq}))^2 
     +
       M_{\gas}(s_{\gas,\DE}^{n})(\de(p_{\gas}))^2.
\end{multline*}
Then, we obtain the following equality
\begin{align}\label{E2}
  E_2 = \frac{1}{2}\sn\pas \sum_{\Du \in \D_\h} \sum_{\Eu \in N(\D)} \tensorDEn 
        \Big( 
             M_{\liq}(s_{\liq,\DE}^{n})(\de(p_\liq))^2 +
             M_{\gas}(s_{\gas,\DE}^{n})(\de(p_\gas))^2
        \Big).
\end{align}
In order to estimate $E_3$, using the fact that the densities are
bounded and the map $g_\alpha$ is sublinear $(\text{a.e.}|g(p_\alpha)|\le
C |p_\alpha|)$, we have
\begin{equation*}
  \abs{E_3} \le C_1\sn\pas \sum_{\Du \in \D_\h} \abs{\Du} (f_{P,\Du}^{n}+f_{I,\Du}^{n})
  (|p_{\liq,\Du}^{n}|+|p_{\gas,\Du}^{n}|),
\end{equation*}
then
\begin{equation*}
  \abs{E_3} \le C_1\sn\pas \sum_{\Du \in \D_\h} \abs{\Du}(f_{P,\Du}^{n}+f_{I,\Du}^{n})
  (2|p_{\Du}^{n}|+|\bar{p}_{\Du}^{n}|+|\tilde{p}_{\Du}^{n}|).
\end{equation*}
Hence, by the H\"older inequality, we get that
\begin{equation*}
  \abs{E_4} \leq C_2\norm{f_P+f_I}_{L^2(Q_T)} 
  \big(
      \sum_{n=1}^{N}\pas \norm{p^{n}_h}_{L^2({\Omega})}^2
      \big)^{\frac{1}{2}},
\end{equation*}
and, from the discrete Poincar\'e--Friedrichs inequality \cite{Vohralik},
we get
\begin{equation}\label{E4}
  \abs{E_{4}} \leq C_3 \big( \sum_{n=1}^{N}\pas
  \norm{p^{n}_h}_{\X_h}^2\big)^{\frac{1}{2}}+C_4.
\end{equation}
The equality \eqref{disc_estimation} with the inequalities \eqref{E1},
\eqref{E2}, \eqref{E4} give \eqref{est:p_alpha}. Then we
deduce \eqref{est:p_lobale} from \eqref{ineq:pression_globale} and Lemma \ref{coercive}.
\end{proof}

We now state the following corollary, which is essential for the
compactness and limit study.
\begin{cor}\label{cor:est} From the previous Proposition, we deduce
  the following estimations:
  \begin{align}
    &\sn \dt \norm{\B(s_{\liq,h}^n)}_{\X_\h}^2 \le C,\label{est:discrete_beta}\\ &
    \sm  \tensorDE M_{\liq,\Du|\Eu}^{n}(\de(\bar{p}(s_\liq)))^2 \le C,
    \label{est:discrete_pbar}
    \end{align}
    and
\begin{align}
  \sm \tensorDE M_{\gas,\Du|\Eu}^{n}(\de(\tilde{p}(s_\liq)))^2 \le
  C.\label{est:discrete_ptilde}
  \end{align}
\end{cor}
\begin{proof}
  The prove of the estimates \eqref{est:discrete_beta},
  \eqref{est:discrete_pbar} and \eqref{est:discrete_ptilde} are a
  direct consequence of the inequality \eqref{ineq:discrete_beta},
  \eqref{ineq:discrete_pbar}, \eqref{ineq:discrete_ptilde}, the Lemma \ref{coercive} and the
  Proposition \ref{prop:estimation_pression}.
\end{proof}

\subsection{Existence of the finite volume scheme}\label{sec:existence}
 
%

%
%
%
%
\begin{prop}
\label{prop:existance}
The problem \eqref{eq:pl disc}-\eqref{eq:pg disc} admits at least one solution
$(p^{n}_{\liq,\Du},p^{n}_{\gas,\Du})_{(\Du ,n) \in \D_\h \times \{1,\cdots,N\}}$.
\end{prop}
The proof is based on  a technical assertion to characterize the zeros of a vector field which stated and proved in \cite{evans:book}.  This method is used in \cite{ZS11fv} and \cite{saadsaadFV}, so it is easy to adopt their proof in our case, thus we omit it.

\section{Compactness properties}\label{sec:compacity}
In this section we derive estimates on differences of space and time
translates of the function $\U_{\alpha,\hdt} = \phi \rho_\alpha(p_{\alpha,\hdt}) s_{\alpha,\hdt}$
which imply that the sequence $\phi \rho_\alpha(p_{\alpha,\hdt})s_{\alpha,\hdt}$
is relatively compact in $L^1(\Qt)$.



The following important relation between $\U_{\dt,\h}$ and $\tilde{\U}_{\dt,\h}$ (see definition \ref{approximationsolution})  is valid:

\begin{lemma} \label{relationsolution} 
                          $\left(
                                    \text{Relation between } \U_{\alpha, \dt,\h} \text{ and }  \tilde{\U}_{\alpha, \dt,\h}
                                    \right).$
	There holds 
	$$
		\norm{\U_{\alpha,\dt,\h} - \tilde{\U}_{\alpha,\dt,\h}}_{L^1(\Qt)} \tend 0 \text{ as } \h \to 0.
	$$
\end{lemma}
\begin{proof}
	\begin{equation*}
      \begin{split}
         \norm{\U_{\liq,\dt,\h} - \tilde{\U}_{\liq,\dt,\h}}_{L^1(\Qt)}   & = 
                        \int_{\Qt}        
                           	  \abs{U_{\liq,\hdt}(t,x) - \tilde U_{\liq,\hdt}(t,x)}\dd x \; \dd t\\
        & \le 
              	\int_{\Qt} \abs{ s_{\liq,\hdt}(t,x) 
				          \Big(
		                                  \rho_\liq(p_{\liq,{\hdt}}(t,x)) - \rho_\liq(\tilde p_{\liq,{\hdt}}(t,x))
		                            \Big) 
		                                  } \dd x \; \dd t \\
	& +  
              	\int_{\Qt} \abs{ \rho_\liq(\tilde p_{\liq,{\hdt}}(t,x))
					  \Big(
					          s_{\liq,\hdt}(t,x) -  \tilde s_{\liq,\hdt}(t,x)
          				  \Big) 
				  	         }\dd x\; \dd t \\
       & \le \mathtt T_1+ \mathtt T_2, 
        \end{split}
\end{equation*}
where $\mathtt T_1$ and $\mathtt T_2$ defined as follows
\begin{equation}\label{note:T1}
  \mathtt T _1 = \rho_M 
  				\int_{\Qt} \abs{ 
					          s_{\liq,\hdt}(t,x) -  \tilde s_{\liq,\hdt}(t,x)
				  	         }\dd x\; \dd t,			
\end{equation}
\begin{equation}\label{note:T2}
\mathtt T_2 = \int_{\Qt} \abs{
		                                  \rho_\liq(p_{\liq,{\hdt}}(t,x)) - \rho_\liq(\tilde p_{\liq,{\hdt}}(t,x))
		                                  } \dd x \; \dd t.
\end{equation}
To handle the term on saturation $\mathtt T_1$, 
we use the fact that $\B^{-1}$ is a H\"older function, then 
$$
\mathtt T_1 \le  \rho_M \cte
  				\int_{\Qt} \abs{ 
					          \B(s_{\liq,\hdt}(t,x)) -  \B(\tilde s_{\liq,\hdt}(t,x))
				  	         }^\theta \dd x\; \dd t,
$$
and by application of the Cauchy-Schwarz inequality, we deduce
$$
\mathtt T_1 \le   \cte \Big(
  				\int_{\Qt} \abs{ 
					          \B(s_{\liq,\hdt}(t,x)) -  \B(\tilde s_{\liq,\hdt}(t,x))
				  	         } \dd x\; \dd t
			         \Big)^\theta
	           \le \cte \; (\mathtt T_1^\prime)^\theta,
$$
where $\mathtt T_1^\prime$ defined as follows 
$$
	T_1^\prime = \int_{\Qt} \abs{ 
					          \B(s_{\liq,\hdt}(t,x)) -  \B(\tilde s_{\liq,\hdt}(t,x))
				  	         } \dd x\; \dd t.
$$
We have
\begin{equation} \label{T1prime}
	\begin{split}
	T_1^\prime 
		          & =  \sn \dt \sk   \sum_{\sigma_\Du \in \E_\K}   
		          	    \int_{\K\cap\Du} \abs{    
   			    					    \B(s_{\liq,\hdt}(t,x)) -  \B(\tilde s_{\liq,\hdt}(t,x))
			  			                } \dd x  \\
			& =  \sn \dt \sk   \sum_{\sigma_\Du \in \E_\K}   
		          	    \int_{\K\cap\Du} \abs{    
   			    					    \B (s_{\liq,\hdt}(t,x)) -  \B(s_{\liq,\hdt}(t,\Q_\Du))
			  			                } \dd x \\
			& =  \sn \dt \sk   \sum_{\sigma_\Du \in \E_\K}   
		          	    \int_{\K\cap\Du} \abs{    
   			    					   \nabla \B (s_{\liq,\hdt}(t,x)) \cdot (x - \Q_\Du) 
			  			                } \dd x  \\
                         & \le  \sn \dt \sk   \sum_{\sigma_\Du \in \E_\K}   
   			    					 \Big| \nabla \B (s_{\liq,\hdt})|_\K\Big|
								   \diam(\Du) \abs{\K\cap \Du}	\\
			& \le  \h \sn \dt \sk  
   			    			 \Big|  \nabla \B (s_{\liq,\hdt})|_\K \Big|
						    \abs{\K}	 \\
	                 & \le \h \left(  \sn \dt \norm{\B(s_{\liq,\hdt})}^2_{\X_\h}
	                                       + \cte
	                 		  \right) \le \cte \h ,					    				   				
   	\end{split}
\end{equation}
where we have used the definitions of $s_{\liq,\hdt}$ 
and $\tilde s_{\liq,\hdt}$, the Cauchy-Schwarz inequality 
and the estimate \eqref{est:discrete_beta}, thus
\begin{equation} 
 \mathtt T_1 \le C \h^\theta. 
\end{equation}

To treat $\mathtt T_2$, we use the fact that the map
$\rho_\liq^\prime$ is bounded and the relationship between the gas
pressure and the global pressure, namely : $p_\liq=p-\bar{p}$ defined
in \eqref{def:pression_globale}, then we have
\begin{equation}\label{e2trans}
   \begin{split}
\mathtt T_2 &\le  \max_\R |\rho_\liq^\prime |\int_{\Qt}
       \abs{p_{\liq,{\hdt}}(t,x) - \tilde p_{\liq,{\hdt}}(t,x)}\dd x\dd t\\
         & \le  
          \max_\R |\rho_\liq^\prime  | \int_{\Qt}
       \abs{p_{\hdt}(t,x) - \tilde p_{\hdt}(t,x)}\dd x\dd t \\
       &+ 
         \max_\R |\rho_\liq^\prime  | \int_{\Qt}\abs{\bar{p}(s_{\liq,\hdt}(t,x)) -\bar{p}( \tilde s_{\liq,\hdt}(t,x)) }\dd x\dd t,
        \end{split}
\end{equation}
furthermore one can easily show that $\bar{p}$ is a $C^1([0,1];\R)$,
it follows, there exists a positive constant $C>0$ such that
\begin{equation*}
      \begin{split}
        \mathtt T_2&\le C \int_{\Qt}
            |p_{\hdt}(t,x) - \tilde p_{\hdt}(t,x) | \dd x \dd t
         + C \int_{\Qt}
        |s_{\liq,\hdt}(t,x) -  \tilde s_{\liq,\hdt}(t,x)| \dd x \dd t.
\end{split}
\end{equation*}
The last term in the previous inequality is proportional to $\mathtt T_1$, and
consequently it remains to show that the term on the global
pressure is small with $\h$. In fact, follows \eqref{T1prime}, one gets
\begin{equation*}
      \begin{split}
           \int_{\Qt}  |p_{\hdt}(t,x) - \tilde p_{\hdt}(t,x) | \dd x \dd t 
               & =  \sn \dt \sk   \sum_{\sigma_\Du \in \E_\K}
               		\int_{\K\cap\Du} \abs{ p_{\hdt}(t,x) - \tilde p_{\hdt}(t,x) } \dd x \\ 
	      & =   \sn \dt \sk   \sum_{\sigma_\Du \in \E_\K}
               		\int_{\K\cap\Du} \abs{ p_{\hdt}(t,x) - p_{\hdt}(t,\Q_\Du) } \dd x \\
              & =   \sn \dt \sk   \sum_{\sigma_\Du \in \E_\K}
               		\int_{\K\cap\Du} \abs{ \nabla p_{\hdt}(t,x) \cdot (x-\Q_\Du) } \dd x \\
	     & \le \h \Big( \sn \dt \norm{p_{\hdt}}^2_{\X_\h} + \cte \Big) \le \cte \h .
	\end{split}
\end{equation*}
Finally, we consider the case where $\alpha = \gas$ in the same manner.
\end{proof}

We now give the space translate estimate for $\tilde  \U_{\alpha,\hdt}$ given by \eqref{sol:FV}.

\begin{lemma}\label{lem:translater-espace}$\left(\text{Space translate
      of }   \tilde \U_{\alpha,\hdt}\right)$. Under the assumptions
  $({A}\ref{hyp:A1})-({A}\ref{hyp:A7})$ . Let $p_{\alpha,\hdt}$ be a solution of
  \eqref{cd:initial discret}--\eqref{eq:pc disc}. Then, the following inequality
  hold:
  \begin{equation}\label{eq:translater_espace}
    \int_{\Omega^{'}\times (0,T)}\abs{ \tilde \U_{\alpha,\hdt}(t,x+y) - 
                                       \tilde \U_{\alpha,\hdt}(t,x)}  
                                    \dd x \dd t\le \omega(\abs{y}),
\end{equation}
for all $y \in \R^\dm$ with $\Omega '=\{x \in \Omega, \, [x,x+y]\subset \Omega\}$ and $\omega(\abs{y}) \to 0$ when $\abs{y}\to 0$.
\end{lemma}
\begin{proof}
For $\alpha = \liq$ and from the definition of $\U_{\liq,\hdt}$, one gets

\begin{equation*}
      \begin{split}
        &\int_{(0,T)\times\Omega^{'}}\abs{\tilde \U_{\liq,\hdt}(t,x+y) - \tilde \U_{\liq,\hdt}(t,x)}
        \dd x\dd t\\
        & = \int_{(0,T)\times\Omega^{'}}
        \abs{\Big(\rho_\liq(\tilde p_{\liq,{\hdt}}) \tilde s_{\liq,\hdt}\Big)(t,x+y) -
          \Big(\rho_\liq(\tilde p_{\liq,{\hdt}}) \tilde s_{\liq,\hdt}\Big)(t,x)} \dd x\dd t\\
        & \le E_1+E_2,
        \end{split}
\end{equation*}
where $E_1$ and $E_2$ defined as follows
\begin{equation}\label{e1trans}
E_1 = \rho_M \int_{(0,T)\times\Omega^{'}} \abs{\tilde s_{\liq,\hdt}(t,x+y)  - \tilde s_{\liq,\hdt}(t,x) }\dd x\dd t,
\end{equation}
\begin{equation}\label{e2transs}
E_2 = \int_{(0,T)\times\Omega^{'}}
        \abs{ \rho_\liq(\tilde p_{\liq,{\hdt}}(t,x+y)) -\rho_\liq(\tilde p_{\liq,{\hdt}}(t,x)) }\dd x\dd t.
\end{equation}
To handle with the  space translation on saturation, we use again  the fact that $\B^{-1}$ is a H\"older function, then 
$$
E_1 \le  \rho_M C \int_{(0,T)\times\Omega^{'}}
        \abs{\B(\tilde s_{\liq,\hdt}(t,x+y) )- \B( \tilde s_{\liq,\hdt}(t,x))}^\theta \dd x\dd t
$$
and by application of the Cauchy-Schwarz inequality, we deduce
$$
E_1\le  C \Big(\int_{(0,T)\times\Omega^{'}}
        \abs{\B(\tilde s_{\liq,\hdt}(t,x+y) )- \B(\tilde s_{\liq,\hdt}(t,x))} \dd x\dd t\Big)^\theta. 
$$

According to \cite{Eymard:book}), let $y \in \R^\dm$, $x \in \Omega '$, and $L \in N(K)$. We define
a function $\beta_\sigma(x)$ for each $\sigma \in \F_\h^\ant$ by  
$$
\beta_{\sigma}=
\begin{cases}
  1, & \text{if the line segment $[x,x+y]$ intersects $\sigma$},\\
  0, & \text{otherwise}.
\end{cases}
$$
We observe that (see  \cite{Eymard:book} for more details)
$
    \int_{\Omega '}\beta_{\sigma_{\DE}}(x) \,dx \le |\sigma_{\DE}|.
  $ 
Now, denote that
\begin{equation*}
      \begin{split}
        E_1 \le & C \Big(\sn \dt \sum_{\sigma_{\DE} \in \F_\h^\ant} 
              \Big|\B(s_{\liq,\Eu}) - \B(s_{\liq,\Du})\Big|  
               \int_{\Omega '}\beta_{\sigma_{\DE}}(x)\dd x\Big)^\theta \\
        & \le 
            C \Big(\abs{y} \sn \dt \sum_{\sigma_{\DE}  \in \F_\h^\ant} \abs{\sigma_{\DE}}
              \Big|\B(s_{\liq,\Eu}) - \B(s_{\liq,\Du})\Big|\Big)^\theta .
\end{split}
\end{equation*}
Let us write $\abs{\sigma_{\DE}} = (d_{\DE}|\sigma_{\DE}|)^{\frac 1 2} (\frac{|\sigma_{\DE}|}{d_{\DE}})^{\frac 1 2}$. Obviously 
$\dis \le \frac{\diam(\K_\DE)}{\dm}$, and $\abs{\sig} \le \frac{\diam(\K_\DE)^{\dm-1} }{\dm-1}$, 
thus  by the regularity shape assumption \eqref{eq:regularity mesh}, we have
\begin{align}\label{regularity}
  \exists \; C_{te} > 0 , \quad \forall \h, \; \forall \Du \in \D_\h  \; \forall \Eu \in \Ne(\Du) \qquad 
  \abs{\sig} \dis\le C_{te} \abs{\K}.
\end{align}
Applying again the Cauchy-Schwarz inequality, using \eqref{regularity}, \eqref{ineq2}
 and the fact that the discrete gradient of the function $\B$ 
is bounded \eqref{est:discrete_beta} to obtain
\begin{equation} \label{e1trans2}
 E_1 \le C \abs{y}^\theta. 
\end{equation}

To treat the space translate of $E_2$, we use the fact that the map
$\rho_\liq^\prime$ is bounded and the relationship between the gas
pressure and the global pressure, namely : $p_\liq=p-\bar{p}$ defined
in \eqref{def:pression_globale}, then we have
\begin{equation*}
   \begin{split}
E_2 &\le  \max_\R |\rho_\liq^\prime |\int_{(0,T)\times\Omega^{'}}
       \abs{\tilde p_{\liq,{\hdt}}(t,x+y) - \tilde p_{\liq,{\hdt}}(t,x)}\dd x\dd t\\
         & \le  
          \max_\R |\rho_\liq^\prime  | \int_{(0,T)\times\Omega^{'}}
       \abs{\tilde p_{\hdt}(t,x+y) - \tilde p_{\hdt}(t,x)}\dd x\dd t \\
       &+ 
         \max_\R |\rho_\liq^\prime  | \int_{(0,T)\times\Omega^{'}}\abs{\bar{p}(\tilde s_{\liq,\hdt}(t,x+y)) -\bar{p}( \tilde s_{\liq,\hdt}(t,x)) }\dd x\dd t,
        \end{split}
\end{equation*}
furthermore one can easily show that $\bar{p}$ is a $C^1([0,1];\R)$,
it follows, there exists a positive constant $C>0$ such that
\begin{equation*}
      \begin{split}
        E_2&\le C \int_{(0,T)\times\Omega^{'}}
            |\tilde p_{\hdt}(t,x+y) - \tilde p_{\hdt}(t,x) | \dd x \dd t\\ 
        & \qquad + C \int_{(0,T)\times\Omega^{'}}
        |\tilde s_{\liq,\hdt}(t,x+y) -  \tilde s_{\liq,\hdt}(t,x)| \dd x \dd t.
\end{split}
\end{equation*}
The last term in the previous inequality is proportional to $E_1$, and
consequently it remains to show that the space translate on the global
pressure is small with $y$. In fact
\begin{equation*}
      \begin{split}
           \int_{(0,T)\times\Omega^{'}}  |\tilde p_{\hdt}(t,x+y) - \tilde p_{\hdt}(t,x) | \dd x \dd t 
               & \le \sn \dt \sum_{\sigma_{\DE}} |p_{\Eu}^{n} - p_{\Du}^{n} | 
            \int_{\Omega '}\beta_{\sigma_{\DE}}(x)\dd x \\
            & \le \abs{y} \sn \dt \sum_{\sigma_{\DE}} \abs{\sigma_{\DE}}
               |p_{\Eu}^{n} - p_{\Du}^{n}|.
\end{split}
\end{equation*}
Finally, using \eqref{ineq2} and the fact that the discrete gradient of global pressure
is bounded \eqref{est:p_lobale}, we deduce that
 \begin{equation}\label{est:space:4}
      \begin{split}
        &\int_{(0,T)\times\Omega^{'}}\abs{\tilde \U_{\liq,\Du}(t,x+y)-\tilde \U_{\liq,\Du}(t,x)}
        \dd x \le C (\abs{y} +\abs{y}^\theta),
\end{split}
\end{equation}
for some constant $C>0$.\\
In the same way, we prove the space translate for $\alpha = \gas$.
\end{proof}

We give below a time translate estimate for $\tilde \U_{\alpha,\hdt}$  given 
by \eqref{sol:FV}.

\begin{lemma}\label{lem:translater-time}
$\left(\text{Time translate of } \tilde \U_{\alpha,\hdt}\right)$. Under the assumptions $({A}\ref{hyp:A1})-({A}\ref{hyp:A7})$ . 
Let $p_{\alpha,\hdt}$ be a solution of
  \eqref{cd:initial discret}--\eqref{eq:pc disc}. Then, there exists a positive
  constant $C>0$ depending on $\Omega$, $T$ such that the following
  inequality hold:
  \begin{equation}\label{eq:translater_temps}
    \int_{\Omega \times
      (0,T-\tau)}\abs{\tilde \U_{\alpha,\hdt}(t+\tau,x) - \tilde \U_{\alpha,\hdt}(t,x)}^2\,dx \,dt \le \tilde{\omega}(\tau),
\end{equation}
for all $\tau\in (0,T)$. Here $\tilde{\omega}:\R^+\to\R^+$ is a modulus
of continuity, i.e. $\lim_{\tau\to 0}\tilde{\omega}(\tau)=0$. 
\end{lemma}

We state without proof the following lemma on time translate of $\U_{\alpha,\hdt}$. 
Following \cite{Eymard:book} and \cite{Eymard-Hilhorst-Vohralik}, the proof is a 
direct consequence of \eqref{ineq2} and the estimations \eqref{est:p_lobale} 
and \eqref{est:discrete_beta}, then we omit it.

\section{Convergence and study of the limit}\label{sec:limite}
Using the a priori estimates of the previous section 
and the Kolmogorov relative compactness theorem, we show in this 
section that the approximate solutions $p_{\alpha,\h,\dt}$ converge
strongly in $L^1(\Qt)$ to a function $p_\alpha$ and we prove that $p_{\alpha}$
 is a weak solution of the continuous problem.
\subsection{Strong convergence in $L^1(\Qt)$ and convergence almost everywhere in $\Qt$}
\begin{theorem}\label{theo:strong-convergence} (Strong convergence in $L^1(\Qt)$)
There exist subsequences of   $s_{\alpha,\hdt}$, $p_{\alpha,\hdt}$, 
$\tilde s_{\alpha,\hdt}$ and  $\tilde p_{\alpha,\hdt}$ verify the following convergence 
\begin{align}   
    &\tilde \U_{\alpha,\hdt} \text{ and } \U_{\alpha,\hdt} \longrightarrow \U_\alpha  && \text{ strongly in }L^1(Q_T) \text{ and a.e. in } Q_T,\label{conv:U}\\
    &\tilde s_{\alpha,\hdt} \text{ and } s_{\alpha,\hdt}\longrightarrow s_\alpha && \text{ almost everywhere in } 
    Q_T,\label{conv:s_alpha}\\
    &\tilde p_{\alpha,\hdt} \text{ and } p_{\alpha,\hdt} {\longrightarrow} p_\alpha && \text{ almost everywhere in }   Q_T. \label{conv:p_alpha} 
\end{align}
Furthermore, $\B(s_\alpha)$ and $p_\alpha$ belongs in $L^1(0,T;H^1_{\Gamma_\liq}(\Omega))$ and 
\begin{align}
\label{conv:s_12}
& 0\le s_\alpha \le 1 \text{ a.e. in } Q_T,\\
& \U_\alpha=\phi \rho_\alpha(p_\alpha)s_\alpha \text { a.e. in } Q_T.
\label{iden=U}
\end{align} 
\end{theorem}
\begin{proof}
  Observe that from Lemma
  \ref{lem:translater-espace} and \ref{lem:translater-time} and Kolmogorov's compactness
 criterion (\cite[Theorem IV.25]{Brezis}, \cite[Theorem 14.1]{Eymard:book}), 
 we deduce that $\tilde \U_{\alpha,\hdt}$ is
 relatively compact in $L^1(\Qt)$.  This ensures the following strong convergences
 of a subsquence of   $\tilde \U_{\alpha,\hdt}$
\begin{align*}
  & \rho_\alpha(\tilde p_{\alpha,\hdt}) \tilde s_{\alpha,\hdt} \longrightarrow l_\alpha
  \quad \text{ in $L^1(Q_T)$ and a.e. in $Q_T$ },
\end{align*} 
and due to the Lemma \ref{relationsolution} , we deduce that $\U_{\alpha,\hdt}$ converges to the 
same $l_\alpha$.

Denote by $u_\alpha= \rho_\alpha(p_\alpha)s_\alpha$. Define the map
$\mathbb{A} : \R^+ \times \R^+ \mapsto \R^+ \times [0,\B(1)]$ defined
by
\begin{equation}
  \mathbb{A}(u_\liq,u_\gas) = (p,\B(s_\liq))
\label{def:H}
\end{equation}
where $u_\alpha$ are solutions of the system
\begin{align*}
  & u_\liq(p,\B(s_\liq)) = \rho_\liq(p-\bar{p}(\B
  ^{-1}(\B(s_\liq))))\B ^{-1}(\B(s_\liq))\\ & u_\gas(p,\B(s_\liq)) =
  \rho_\gas(p-\tilde{p}(\B ^{-1}(\B(s_\liq))))(1-\B ^{-1}(\B(s_\liq)).
\end{align*}
Note that $\mathbb{A}$ is well defined as a diffeomorphism \cite{CS07}, \cite{ZS10} and \cite{CS10}. 
As the map ${\mathbb{A}}$ defined in \eqref{def:H} is continuous, we
deduce
\begin{align*}
  & p_\hdt \longrightarrow p \quad \text{ a.e. in } Q_T,\\
  & \B(s_{l,\hdt}) \longrightarrow \B^{*}\quad \text{ a.e. in } Q_T.
\end{align*} 
Then, as $\B^{-1}$ is continuous, this leads to the desired estimate  \eqref{conv:s_alpha} 
$$
s_{l,\hdt} \longrightarrow s_l = \B^{-1}(\B^{*})\quad \text{ a.e. in }
  Q_T.
$$

Consequently and due to the relationship between the pressure of each
phase and the global pressure defined in \eqref{def:pression_globale},
then the convergences \eqref{conv:p_alpha} hold  
\begin{align*}
  & p_{\alpha,\hdt} \longrightarrow p_\alpha \quad \text{ a.e. in } Q_T.
\end{align*} 

Moreover, due to the space translate estimate on the saturation and the global pressure
\eqref{e1trans}-\eqref{e2transs}, \cite[Theorem 3.10]{Eymard:book} gives that
$\B(s_\alpha)$ and $p\in L^1(0,T;H^1_{\Gamma_\liq}(\Omega))$.
The identification of the limit in \eqref{iden=U} follows from the
previous convergence.
\end{proof}


\subsection{Proof of theorem \ref{theo:principal}}
In order to achieve the proof of Theorem \ref{theo:principal} and show that 
$p_\alpha$ is a weak solution of the continuous problem, it remains to pass to the limit 
as $(\hdt)$ goes to zero in the formulations \eqref{eq:pl disc}--\eqref{eq:pg disc}. 
For this purpose, we introduce 
\begin{align}
\mathcal G := \{ \psi  \in \mathcal C^{2,1}(\Omega \times [0,T]), \psi = 0 \text{ on }  
\partial \Omega \times [0,T], \psi(.,T) = 0  \} .
\end{align}
  Let $T$ be a fixed positive constant and $\psi \in \mathcal G$. 
  Set $\psi_\Du^n:=\psi(t^n,\Q_\Du)$ for all
  $\Du \in \D_\h$ and $n\in[0,N]$.\\
  For the discrete liquid equation, we multiply the equation
  \eqref{eq:pl disc} by $\pas \psi_\Du^{n}$ and sum 
  the result  over $\Du \in \D_\h^\ant$
  and $n\in \{1,\cdots,N\}$. This yields
  $$
  \cc_1 +\cc_2+\cc_3  = 0,
  $$
  where 
  \begin{equation*}
    \begin{split}
      \cc_1 & = \sn  \sum_{\Du \in \D_\h} \abs{\Du}\phi_\Du \left(
        \rho_\liq(p^{n}_{\liq,\Du})s^{n}_{\liq,\Du}
        -\rho_\liq(p^{n-1}_{\liq,\Du})s^{n-1}_{\liq,\Du}\right) \psi_\Du^{n},\\
      \cc_2 & = \sm \rho^{n}_{\liq,\DE} \tensorDE G_\liq(s^{n}_{\liq,\Du},s^{n}_{\liq,\Eu};\de(p_{\liq})) \psi_\Du^{n},\\
      \cc_3 & = \sm \abs{\Du} \left( \rho_\liq(p_{\liq,\Du}^{n})
        s_{\liq,\Du}^{n} f_{P,\Du}^{n}\psi_\Du^{n}-
        \rho_\liq(p_{\liq,\Du}^{n}) \sli
        f_{I,\Du}^{n}\psi_\Du^{n}\right).
    \end{split}
  \end{equation*}
  using $\psi_\Du^{n} = 0$ for all $\Du \in \D_\h^\ext$ and  $n\in \{0,\cdots,N\}$. 
  We now show that each of the above terms converges to its continuous version as $\h$ and $\dt$ 
  tend to zero. \\

 Firstly, for the evolution term. Making summation by parts in time and keeping in mind that
  $\psi(T=t^{N},\Q_\Du) = \psi_\Du^{N}=0$. For all $\Du \in \D_\h$, we
  get
  \begin{equation*}
    \begin{split}
      \cc_1 = & - \sn \sum_{\Du \in \D_\h}  \abs{\Du} \phi_\Du \rho_\liq(p^{n}_{\liq,\Du}) s^{n}_{\liq,\Du}
                                    \left(
 		                         \psi_\Du^{n}- \psi_\Du^{n-1} 
		                 \right)  
                         - \sum_{\Du \in \D_\h} \abs{\Du} \phi_\Du \rho_\liq(p^{0}_{\liq,\Du})s^{0}_{\liq,\Du} \psi_\Du^0 \\
                   = & -\sn \sum_{\Du \in \D_\h} 
                                                   \int_{t^{n-1}}^{t^{n}}  
                                                       \int_{\Du}  
                                                           \phi_\Du \rho_\liq(p^{n}_{\liq,\Du}) s^{n}_{\liq,\Du} \partial_t \psi(t,\Q_\Du)\dd x\dd t 
                        - \sum_{\Du \in \D_\h} \int_{\Du} 
                                                                   \phi_\Du \rho_\liq(p^{0}_{\liq,\Du})s^{0}_{\liq,\Du}\psi(0,\Q_\Du) \dd x.
    \end{split}
  \end{equation*}
Since $\phi_\h \rho_\liq(p_{\liq,\hdt}) s_{\liq,\hdt}$ and $\phi_\h \rho_\liq(p^0_{\liq,\hdt})s^0_{\liq,\hdt}$ 
converge almost everywhere respectively to $\phi\rho_\liq(p_\liq)s_\liq$ and $\phi
\rho_\liq(p^0_\liq)s^0_\liq$, and as a consequence of Lebesgue dominated
convergence theorem, we get
\begin{equation*}
  \cc_1 \tend - \int_{Q_T}\phi \rho_\liq(p_\liq)s_\liq \partial_t \psi(t,x) \dd x \dd t 
  - \int_{\Omega}\phi \rho_\liq(p^0_\liq)s^0_\liq \psi(0,x)\dd x, \text{ as }  \h, \dt \to 0.
\end{equation*}
Now, let us focus on convergence of the degenerate diffusive term to show 
\begin{equation}\label{conv:S2h}
  \begin{split}
    \cc_2 \tend - \int_{Q_T}\rho_\liq(p_\liq)M_\liq(s_\liq)\nabla p_\liq\cdot \nabla \psi
    \dd x \dd t, \text{ as }  \h, \dt \to 0.
  \end{split}
\end{equation}
Since the discrete gradient of each phase is not bounded, it is not possible to justify the pass to the limit in a straightforward way. 
To do this, we use the feature of global pressure and the auxiliary pressures defined in \eqref{def:pression_globale} 
and the discrete energy estimates in proposition \ref{prop:estimation_pression} and corollary \ref{cor:est}.

We rewrite $\cc_2$ as 
\begin{align*}
	\cc_2 =  \cc_{2,1}+\cc_{2,2}
\end{align*}
with, by using the definition \eqref{def:pression_globale}, 
  \begin{align*}
    & \cc_{2,1} = - \sm \tensorDE  
    \rho^{n}_{\liq,\DE}M_\liq(s^{n}_{\liq,\DE}) \de(p)    \psi_\Du^{n}, \\
    & \cc_{2,2} =  \sm \tensorDE
    \rho^{n}_{\liq,\DE}M_\liq(s^{n}_{\liq,\DE}) \de(\bar{p}(s_\liq))  \psi_\Du^{n}.
  \end{align*}
Let us show that 
\begin{align}\label{conv:A1}
  \cc_{2,1} \tend - \int_{Q_T} \tensor(x) \; \rho_\liq(p_\liq)M_\liq(s_\liq)
  \nabla p \cdot \nabla \psi\, \dd x \dd t \text{ as } \hdt \to 0.
\end{align}
For each couple of neighbours $\Du$ and $\Eu$ we denote $\slmin$ the minimum of
$s_{\liq,\Du}^{n}$ and $s_{\liq,\Eu}^{n}$ and we introduce
\begin{align} \label{A1*}
  \cc_{2,1}^{*} = - \sm \tensorDE \;  \rho^{n}_{\liq,\DE}M_\liq(\slmin)
  \de(p)  \psi_\Du^{n}
\end{align}

We now show %
\begin{align}\label{conv:A*}
   \cc_{2,1}^{*} \tend - \int_{Q_T} \tensor(x) \; \rho_\liq(p_\liq)M_\liq(s_\liq)
  \nabla p \cdot \nabla \psi\, \dd x \dd t \text{ as } \hdt \to 0
\end{align}
as $\hdt \to 0$. Define $\overline{s}_{\alpha,\hdt}$ and $\underline{s}_{\alpha,\hdt}$
by
$$
\overline{s}_{\alpha,\hdt}|_{(t^n,t^{n}]\times \K_{\DE}}:=\max\{s_{\alpha,\Du}^{n},s_{\alpha,\Eu}^{n}\},\quad
\underline{s}_{\alpha,\hdt}|_{(t^n,t^{n}]\times \K_{\DE}}:=\min\{s_{\alpha,\Du}^{n},s_{\alpha,\Eu}^{n}\}
$$
Remark that
\begin{align*}
  \cc_{2,1}^{*} & = - \sm \tensorDE \;  \rho^{n}_{\liq,\DE}M_\liq(\slmin)  \de(p)  \psi_\Du^{n} \\
  		  & = - \snDE \tensorDE \;  \rho^{n}_{\liq,\DE}M_\liq(\slmin)  p_\Eu^{n}  \psi_\Du^{n} \\
		  & = - \snDE   \rho^{n}_{\liq,\DE}M_\liq(\slmin) \sum_{\K \in \T_\h}  \left(
		  									\tensor(x) \nabla \varphi_\Eu, \nabla \varphi_\Du 
															\right)_{0,\K} p_\Eu^{n}  \psi_\Du^{n} \\
		 & = \sn \dt \sum_{\K \in \T_\h} \int_{\K} \tensor(x) \rho_\liq(p_{\liq,\h}^{n}) M_\liq(\underline{s}_{\liq,\h}^{n})
		         \nabla p_\h^{n} \cdot \nabla \left( \sum_{\Du \in \D_\h} \psi(t^{n},\Q_\Du ) \varphi_\Du(x) \right) \dd x 					         
\end{align*}
We will show the validity of two passages to the limit. We begin by defining : 
\begin{align*}
  \cd_1  =\cc_{2,1}^{*}- \sn \dt \sum_{\K \in \T_\h} \int_{\K} \tensor(x) \rho_\liq(p_{\liq,\h}^{n}) M_\liq(\underline{s}_{\liq,\h}^{n})
		         \nabla p_\h^{n} \cdot \nabla \psi(t^{n},x ) \dd x.
\end{align*}
We then estimate 
$$
\abs{\cd_1} \le C \h,
$$
using the  estimate \eqref{est:p_lobale}, for more details see \cite[Theorem 15.3]{Ciarlet} and \cite[section 6.2]{Eymard-Hilhorst-Vohralik}. Then, 
$$
\cd_1\tend 0 \text{ as } \h \to 0.
$$

We next show that 
\begin{multline}\label{conv:fin A1}
\sn \dt \sum_{\K \in \T_\h} \int_{\K} \tensor(x) \rho_\liq(p_{\liq,\h}^{n}) M_\liq(\underline{s}_{\liq,\h}^{n})
		         \nabla p_\h^{n} \cdot \nabla \psi(t^{n},x ) \dd x  \tend  \\ 
\int_{0}^{T} \int_\Omega \tensor(x) \rho_\liq(p_\liq) M_\liq(s_\liq) \nabla p(t,x) \cdot \nabla \psi(t,x) \; \dd x \dd t
\end{multline}
as $\hdt \to 0$. We see that both $p_\h^{n}(x)$ and $\psi(t^{n},x)$ are constant in time, so that we can  easily introduce
an integral with respect to time into the first term of \eqref{conv:fin A1}. We further add and subtract 
$$
\sn \int_{t^{n-1}}^{t^{n}} \int_\Omega  \tensor(x) \rho_\liq(p_{\liq,\h}^{n}) M_\liq(\underline{s}_{\liq,\h}^{n})
		         \nabla p_\h^{n} \cdot \nabla \psi(t,x ) \; \dd x \dd t  
$$
and introduce 
\begin{align*}
 & \cd_2  :=   \sn \int_{t^{n-1}}^{t^{n}} \sum_{\K \in \T_\h} \int_{\K} \tensor(x) \rho_\liq(p_{\liq,\h}^{n}) M_\liq(\underline{s}_{\liq,\h}^{n})
		                 \nabla p_\h^{n} \cdot \left( \nabla \psi(t^{n},x) - \nabla \psi(t,x) \right) \; \dd x \dd t, \\ 
& \cd_3  :=  \int_{0}^{T} \sum_{\K \in \T_\h} \int_{\K}	  \tensor(x)  \rho_\liq(p_{\liq,\hdt}) M_\liq(\underline{s}_{\liq,\hdt})
			       \nabla  p_\hdt(t,x) \cdot   \nabla \psi(t,x )  \\ & -   \int_{0}^{T}  \int_\Omega 	 \tensor(x) \; \rho_\liq(p_\liq)M_\liq(s_\liq)
  \nabla p(t,x) \cdot \nabla \psi(t,x) \; \dd x \dd t          
\end{align*}
where $p_\hdt$ is given by \eqref{sol:FE}.  Clearly, \eqref{conv:fin A1} is valid when $\cd_2 $ and $\cd_3 $ tend to zero 
as $\hdt \to 0$. We first estimate $\cd_2 $. We have, for $t \in (t^{n-1},t^{n}]$,
$$
\abs{ \nabla \psi(t^{n},x) - \nabla \psi(t,x)} \le g(\dt),
$$
where $g$ satisfies $g(\dt)>0$ and $g(\dt) \tend 0$ as $\dt \to 0$. Thus  
$$
\abs{\cd_2 } \le C g(\dt) \sn \dt \sum_{\K \in \T_\h} \abs{\nabla p_\h^{n}|_\K} \abs{\K} 
                            \le  C g(\dt) T^{\frac 1 2} \abs{\Omega}^{\frac 1 2}
$$
using the Cauchy-Schwarz inequality and the estimate \eqref{est:p_lobale}. \\

We now turn to $\cd_3$. We easily notice that we cannot use the Green theorem for 
$p_\h^{n}$ on $\Omega$, since  $p_\h^{n}\notin H^1(\Omega)$. 
 So, we are thus forced to apply it on each $\K \in \T_\h$.

To show that  $\cd_3 \tend 0$ as $\hdt \to 0$,
we begin by showing that 
\begin{align}\label{conv-weak:p}
	 \int_{0}^{T} \sum_{\K \in \T_\h} \int_{\K} \left( \nabla p_{\hdt}(t,x) - \nabla p(t,x) \right) \cdot  \W(t,x) \; \dd x \dd t \tend 0
\end{align}
as $\hdt \to 0$ for all $\W \in (\C^1(\overline{\Qt}))^\dm$. To this purpose, 
we use the a priori estimate \eqref{est:p_lobale} and \cite[Section 6.2]{Eymard-Hilhorst-Vohralik}.
Using the density of the set $[\C^1(\overline{\Qt})]^\dm$ in $[L^2(\Qt)]^\dm$ 
and \eqref{conv-weak:p}, we will conclude a weak convergence
of $\nabla p_\hdt$ (piecewise constant function is space and time) to $\nabla p$.

We now finally conclude that $\cd_3\tend 0$ as $\hdt \to 0$. To do that,
we begin by showing that $\underline{s}_{\liq,\hdt} \to s_\liq$ a.e on $Q_T$.
Define $\overline{s}_{\alpha,\hdt}$ and $\underline{s}_{\alpha,\hdt}$
by
$$
\overline{s}_{\alpha,\hdt}|_{(t^n,t^{n}]\times \K_{\DE}}:=\max\{s_{\alpha,\Du}^{n},s_{\alpha,\Eu}^{n}\},\quad
\underline{s}_{\alpha,\hdt}|_{(t^n,t^{n}]\times \K_{\DE}}:=\min\{s_{\alpha,\Du}^{n},s_{\alpha,\Eu}^{n}\}
$$
 By the monotonicity of $\B$, we have
\begin{align*}
  \int_{0}^{T}\int_{\Omega}
  \abs{\B(\overline{s}_{\liq,\hdt})-\B(\underline{s}_{\liq,\hdt})}^2\dd x \dd t \le &
   \sm \int_{\K_\DE} \left(\B(s_{\liq,\Eu}^{n}) - \B(s_{\liq,\Du}^{n}) \right)^2   \dd x\\ &\le 
   \sm \int_{\K_\DE} \abs{ \nabla \B(s_{\liq,\h}^{n})|_{\K_\DE}}^2 \dis^2 \;    \dd x \\ &\le 
   \sm  \abs{ \nabla \B(s_{\liq,\h}^{n})|_{\K_\DE}}^2 \dis^2   \abs{{\K_\DE}}\\ &\le 
 \h^2   \sn \dt \sum_{\sig \in \F_\h^\ant} \abs{ \nabla \B(s_{\liq,\h}^{n})|_{\K_\DE}}^2   \abs{{\K_\DE}}\\ &\le 
 \h^2 \sn \dt \norm{\B(s_{\liq,\h}^n)}_{\X_\h}^2 \le C \h^2
\end{align*}
where we have used the estimate \eqref{est:discrete_beta}.\\
Since $\B^{-1}$ is continuous, we deduce up to a subsequence
\begin{align}\label{conv:sminsmax}
  \abs{\underline{s}_{\alpha,\D_m}-\overline{s}_{\alpha,\D_m}}\to 0
  \text{ a.e. on } Q_T.
\end{align}
Moreover, we have $\underline{s}_{\alpha,\hdt}\le
s_{\alpha,\hdt}\le\overline{s}_{\alpha,\hdt} $ and $s_{\alpha,\hdt}\to
s_\alpha$ a.e. on $Q_T$. Consequently, and due to the continuity of
the mobility function $M_\liq$ we have 
\begin{align}\label{conv:smin}
M_\liq(\underline{s}_{\liq,\hdt})\to M_\liq(s_\liq) 
\end{align}
a.e on $Q_T$ and in $L^p(Q_T)$ for $p<+\infty$.\\
Finally, we further add and subtract
$ \int_{0}^{T}  \int_\Omega 	
   \tensor(x) \; \rho_\liq(p_\liq)M_\liq(s_\liq)
   \nabla p_\hdt(t,x) \cdot \nabla \psi(t,x) \; \dd x \dd t $ 
to $\cd_3$ and using \eqref{conv:p_alpha}, \eqref{conv:smin},
the a priori estimate \eqref{est:p_lobale}, the weak convergence 
of $\nabla p_\hdt$ to $\nabla p$ \eqref{conv-weak:p}, to conclude
that $\cd_3\tend 0$ as $\hdt \to 0$. Altogether, combining 
\eqref{A1*} and \eqref{conv:fin A1}  gives 
\begin{align*}
   \cc_{2,1}^{*} \tend - \int_{Q_T} \tensor(x) \; \rho_\liq(p_\liq)M_\liq(s_\liq)
  \nabla p \cdot \nabla \psi\, \dd x \dd t \text{ as } \hdt \to 0.
\end{align*} 

It remains to show that 
\begin{equation}\label{A1m-A1m*}
  \abs{\cc_{2,1}-\cc_{2,1}^{*}} \tend 0 \text{ as } \hdt \to 0.
\end{equation}
Remark that 
\begin{equation*}
  \abs{M_\liq(s_{\liq,\DE}^{n})\de(p)-M_\liq(\slmin)\de(p)}\le
  C \abs{s_{\liq,\Eu}^{n}-s_{\liq,\Du}^{n}}\abs{\de(p)}.
\end{equation*}
Consequently
$$
  \abs{\cc_{2,1}-\cc_{2,1}^{*}} \le C\sn \dt \sum_{\K \in \T_\h} \int_{\K} \abs{s_{\liq,\Eu}^{n}-s_{\liq,\Du}^{n}} 
		         \nabla p_\h^{n} \cdot \nabla \left( \sum_{\Du \in \D_\h} \psi(t^{n},\Q_\Du ) \varphi_\Du(x) \right) \dd x 	
$$
Applying the Cauchy-Schwarz inequality, and thanks to the uniform
bound \eqref{est:p_lobale} and the convergence
\eqref{conv:sminsmax}, we establish \eqref{A1m-A1m*}.\\


To prove the pass to limit of $\cc_{2,2}$, we need to prove firstly that
\begin{align*}
  \| \de(\Gamma(s_\liq)) - \sqrt{M_\liq(s^{n}_{\liq,\DE})}
  \de(\bar{p}(s_\liq))\|_{L^2(Q_T)} \to 0 \text{ as }
  \hdt \to 0,
\end{align*}
where $\Gamma(s_\liq)=\int_{0}^{s_\liq}\sqrt{M_\liq(z)}\frac{\dd
  \bar{p}}{\dd s_\liq}(z) \dd z$.\\
In fact, remark that there exist $a\in[s_{\liq,\Du},s_{\liq,\Eu}]$ such as:
 \begin{align*}
   |\de(\Gamma(s_\liq)) - \sqrt{M_\liq(s^{n}_{\liq,\Du,\Eu})}\de(\bar{p}(s_\liq))|& 
   = |\sqrt{M_\liq(a)} - \sqrt{M_\liq(s^{n}_{\liq,\Du,\Eu})}|| 
        \de(\bar{p}(s_\liq))|\\ & 
   \le C |\de(\bar{p}(s_\liq))| 
   \le C \abs{s_{\liq,\Eu}^{n}-s_{\liq,\Du}^{n}} \\ &
   \le C \abs{\B(s_{\liq,\Eu}^{n}) - \B(s_{\liq,\Du}^{n})}^\theta,
    \end{align*}
since $\B^{-1}$ is an H\"older function. Thus we get,
 \begin{equation*}
   \begin{aligned}
     & \| \de(\Gamma(s_\liq)) - \sqrt{M_\liq(s^{n}_{\liq,\Du,\Eu})}
     \de(\bar{p}(s_\liq))\|^2_{L^2(Q_T)}\\ &
          =  \sm |\K_\DE| | \de(\Gamma(s_\liq)) - \sqrt{M_\liq(s^{n}_{\liq,\Du,\Eu})}
           \de(\bar{p}(s_\liq))|^2 \\ & 
     \le \sm |\K_\DE|^{1-\theta} |\K_\DE|^{\theta}
         \abs{\B(s_{\liq,\Eu}^{n}) - \B(s_{\liq,\Du}^{n})}^{2\theta},
       \end{aligned}
 \end{equation*}
 and using the Cauchy-Schwarz inequality and the estimate , we deduce 
  \begin{equation*}
   \begin{aligned}
     & \| \de(\Gamma(s_\liq)) - \sqrt{M_\liq(s^{n}_{\liq,\Du,L})}
     \de(\bar{p}(s_\liq))\|^2_{L^2(Q_T)}\\ &
       \le \left(\sm |\K_\DE|\right)^{1-\theta}\left(\sm |\K_\DE|
         \abs{\B(s_{\liq,E}^{n}) - \B(s_{\liq,D}^{n})}^2\right)^{\theta}\\ &
       \le \left(\sm |\K_\DE|\right)^{1-\theta}\left(\sm |\K_\DE|
          \abs{ \nabla \B(s_{\liq,\h}^{n})|_{\K_\DE}}^2 \dis^2\right)^{\theta}\\ &   
    \le C \h^{2\theta} \left(\sn \dt \sum_{\sigma_\DE \in \F_\h^\ant} |\K_\DE|
          \abs{ \nabla \B(s_{\liq,\h}^{n})|_{\K_\DE}}^2 \right)^{\theta} \\ &
          \le C \h^{2\theta} \left(\sn \dt \norm{ \nabla \B(s_{\liq,\h}^{n})}_{\X_\h}^2 \right)^{\theta} \le C \h^{2\theta}
       \end{aligned}
 \end{equation*}
 which shows that $ \| \de(\Gamma(s_\liq)) -
 \sqrt{M_\liq(s^{n}_{\liq,\DE})} \de(\bar{p}(s_\liq))\|^2_{L^2(Q_T)} \to
 0$ as $\h \to 0$.  And from \eqref{est:discrete_pbar} in corollary \ref{cor:est},
 we deduce that there exists a constant $C>0$ where the following
 inequalities hold:
  \begin{align}\label{est:discrete_Gamma}
    \sn \dt  \norm{ \Gamma(s_{\liq,\h}^{n})}_{\X_\h}^2 \le C.
  \end{align}
 %
 Furthermore, by
  \eqref{conv:s_alpha}, we have
  $$
  \Gamma(s_{\liq,\hdt}) \tend \Gamma(s_{\liq}) \text{ in } L^2(\Qt) .
  $$
 In the same manner of 
  \eqref{conv-weak:p}, we prove a weak convergence of $\nabla \Gamma(s_{\liq,\hdt}) $ (piecewise constant function is space and time) to
  $\nabla \Gamma(s_\liq)$.
 As consequence
  $\sqrt{M_\liq(s_{\liq,\hdt})}\nabla  \bar{p}(s_{\liq,\hdt})$ converges to
  $\nabla \Gamma (s_\liq)$ in $L^2(Q_T)$, and 
\begin{align}\label{lim:A2m}
   \cc_{2,2} & \tend  - \int_{0}^{T}\int_{\Omega}
  \rho_\liq(p_{\liq}) \sqrt{M_\liq(s_{\liq})} \nabla \Gamma(s_{\liq})\cdot
  \nabla \psi\dd x \dd t\\ & = - \int_{0}^{T}\int_{\Omega}
  \rho_\liq(p_{\liq}) M_\liq(s_{\liq}) \nabla \bar{p}(s_{\liq})\cdot \nabla
  \psi\dd x \dd t.
\end{align}

 


\end{document}